\documentclass[bjps,preprint]{imsart}

\RequirePackage[OT1]{fontenc}
\RequirePackage{amsthm,amsmath,natbib}
\RequirePackage[colorlinks,citecolor=blue,urlcolor=blue]{hyperref}
\usepackage{amssymb,caption,mathrsfs,enumerate,tikz,relsize,exscale,graphicx}

\arxiv{1509.08172}

\startlocaldefs
\newtheorem{definition}{Definition}[section]
\newtheorem{lemma}[definition]{Lemma}
\newtheorem{theorem}[definition]{Theorem}
\newtheorem{restrict}[definition]{Restriction}
\newtheorem{remark}[definition]{Remark}
\numberwithin{equation}{section}

\newcommand{\R}{\mathbb{R}}
\newcommand{\N}{\mathbb{N}}
\newcommand{\D}{\mathbb{D}}

\newcommand{\bb}[1]{\boldsymbol{#1}}

\newcommand{\floor}[1]{\lfloor #1 \rfloor}
\newcommand{\prob}[2]{\mathbb{P}_{#1}\hspace{-0.3mm}\left( #2 \right)}
\newcommand{\esp}[2]{\mathbb{E}_{#1}\hspace{-0.3mm}\left[ #2 \right]}
\newcommand{\var}[2]{\mathbb{V}_{#1}\hspace{-0.3mm}\left( #2 \right)}
\newcommand{\cov}[2]{\text{Cov}\hspace{-0.3mm}\left(#1,#2\right)}
\endlocaldefs

\begin{document}

\begin{frontmatter}

    \title{Maxima of branching random walks \\with piecewise constant variance}%
    \runtitle{Time-inhomogeneous branching random walk}%

    \begin{aug}
        \author{
            \fnms{Fr\'ed\'eric}%
            \snm{Ouimet}%
            \ead[label=e1]{ouimetfr@dms.umontreal.ca}%
            \ead[label=e2,url]{https://sites.google.com/site/fouimet26}%
        }%
        \affiliation{Universit\'e de Montr\'eal}%
        \address{
            F. Ouimet \\
            D\'epartement de Math\'ematiques et de Statistique \\
            Universit\'e de Montr\'eal \\
            2920 chemin de la Tour \\
            Montr\'eal, Qu\'ebec H3T 1J4 \\
            Canada \\
            \printead{e1} \\
            \printead{e2}
        }%
        \runauthor{F. Ouimet}%
    \end{aug}

    \begin{abstract}
        This article extends the results of \cite{MR2968674} on branching random walks (BRWs) with Gaussian increments in time inhomogeneous environments. We treat the case where the variance of the increments changes a finite number of times at different scales in $[0,1]$ under a slight restriction. We find the asymptotics of the maximum up to an $O_{\mathbb{P}}(1)$ error and show how the profile of the variance influences the leading order and the logarithmic correction term. A more general result was independently obtained by \cite{MR3361256} when the law of the increments is not necessarily Gaussian. However, the proof we present here generalizes the approach of \cite{MR2968674} instead of using the spinal decomposition of the BRW. As such, the proof is easier to understand and more robust in the presence of an approximate branching structure.
    \end{abstract}



\end{frontmatter}

\section{Introduction}
    \subsection{The model}\label{sec:model}

    The tree underlying the branching process we are interested in can be described as follows. At time $k = 0$, there exists only one particle $o$, called the \textit{origin}, and we set $\D_0 \circeq \{o\}$. At time $k = 1$, there are $b = 2$ particles and each of them is linked to $o$ by an edge. Denote by $\D_1$ the set of particles at time $1$. At time $k = 2$, there are four particles, two of which are linked to the first particle in $\D_1$ and the other two are linked to the second particle in $\D_1$. The set of particles at time $2$ is denoted by $\D_2$. The tree is defined iteratively in this manner up to time $k = n$, where $\D_k$ denotes the set of all particles at time $k$ and $|\D_k| = 2^k$. Figure \ref{fig:binary.tree} illustrates the tree structure.

    \setlength{\belowcaptionskip}{-6pt}
    \begin{figure}[ht]
        \centering
        \includegraphics[scale=0.8]{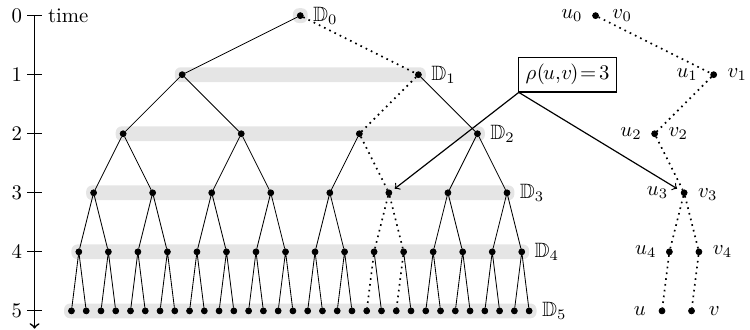}
        \captionsetup{width=0.8\textwidth}
        \caption{The tree structure with a branching factor $b=2$.}
        \label{fig:binary.tree}
    \end{figure}

    For all $v\in \D_n$, we denote by $v_k$ the \textit{ancestor of $v$ at time $k$}, namely the unique particle in $\D_k$ that intersects the shortest path from $o$ to $v$. The \textit{branching time} $\rho(u,v)$ is the latest time at which $u,v\in \D_n$ have the same ancestor. Formally,
    \vspace{-1.5mm}
    \begin{equation*}
        \rho(u,v) \circeq \max\{k\in \{0,1,\ldots,n\} : u_k = v_k\}.
    \end{equation*}

    In the standard branching random walk (BRW) setting, i.i.d.\hspace{-0.3mm} Gaussian random variables $\mathcal{N}(0,\sigma^2)$ are assigned to each branch of the tree structure and the field of interest is $\{S_v\}_{v\in \D_n}$, where $S_v$ is the sum of the Gaussian variables along the shortest path from $o$ to $v$. In the time-inhomogeneous context, the variance of the Gaussian variables depends on time. Fix $M\in \N$ and consider the parameters
    \begin{equation*}
        \begin{aligned}
            \boldsymbol{\sigma} &\circeq (\sigma_1,\sigma_2, \ldots, \sigma_M)\in (0,\infty)^M \ \ \ \ \ \ \ \ \ & \text{(variance parameters)} \\
            \boldsymbol{\lambda} &\circeq (\lambda_1,\lambda_2, \ldots, \lambda_M)\in (0,1]^M & \text{(scale parameters)}
        \end{aligned}
    \end{equation*}
    where $0 \circeq \lambda_0 < \lambda_1 < \ldots < \lambda_M\circeq 1$.
    The parameters $(\boldsymbol{\sigma}, \boldsymbol{\lambda})$ can be encoded simultaneously in the left-continuous step function
    \vspace{-1mm}
    \begin{equation*}
        \sigma(s) \circeq \sigma_1 \boldsymbol{1}_{\{0\}}(s) + \sum_{i=1}^M \sigma_i \boldsymbol{1}_{(\lambda_{i-1},\lambda_i]}(s), \quad s\in [0,1].
    \end{equation*}
    The following definition and the results of this paper are easily extended to BRWs with other \textit{branching factors} $b\in \N$.

    \begin{definition}\label{def:IBRW}
        The $(\bb{\sigma},\bb{\lambda})$-BRW of length $n$ is a collection of positively correlated random walks $\{\{S_v(t)\}_{t=0}^n\}_{v\in \D_n}$ defined by
        \begin{equation}\label{eq:IBRW}
            S_v(t) \circeq \sum_{i=1}^M \sum_{k = \floor{\lambda_{i-1} n} + 1}^{\floor{\lambda_i n} \wedge t} \hspace{-4mm}\sigma_i \hspace{0.5mm}Z_{v_k}, \ \ \ t\in \{0,1,\ldots,n\}, \ v\in \D_n,
        \end{equation}
        where $\{Z_{v_k}\}_{k\in \{1,\ldots,n\};v\in \D_n}$ are i.i.d.\hspace{-0.3mm} $\mathcal{N}(0,1)$ random variables and $b = 2$.
    \end{definition}

    By convention, summations are zero when there are no indices. To avoid trivial corrections in the proofs, always assume, without loss of generality, that $t_i \circeq \lambda_i n\in \N_0$ for all $i\in \{0,1,\ldots,M\}$. Hence, the floor functions can be dropped in \eqref{eq:IBRW}. For simplicity, we set $S_v \circeq S_v(n)$.

    \subsection{Main result}\label{sec:main.result}

    First, we introduce some notations.
    For any positive measurable function $f:[0,1]\to\R$, define the integral operators
    \vspace{-1mm}
    \begin{equation*}
        \mathcal{J}_f(s) \circeq \int_0^s f(r) dr \quad \text{and} \quad \mathcal{J}_f(s_1,s_2) \circeq \int_{s_1}^{s_2} f(r) dr.\vspace{-1mm}
    \end{equation*}
    The first order of the maximum for the $(\bb{\sigma},\bb{\lambda})$-BRW is merely the solution to an optimization problem involving the {\it concave hull} of $\mathcal{J}_{\sigma^2}(\cdot)$, which we denote by $\hat{\mathcal{J}}_{\sigma^2}$.
    We refer the reader to \cite{Ouimet2014master} for a detailed heuristic and a rigorous proof, and to \cite{MR3541850} for the same results in the context of the scale-inhomogeneous Gaussian free field.
    By definition, the graph of $\hat{\mathcal{J}}_{\sigma^2}$ is an increasing and concave polygonal line, see Figure \ref{fig:concave} below for some examples.

    \vspace{4mm}
    \setlength{\belowcaptionskip}{0pt}
    \begin{figure}[ht]
        \includegraphics[scale=0.51]{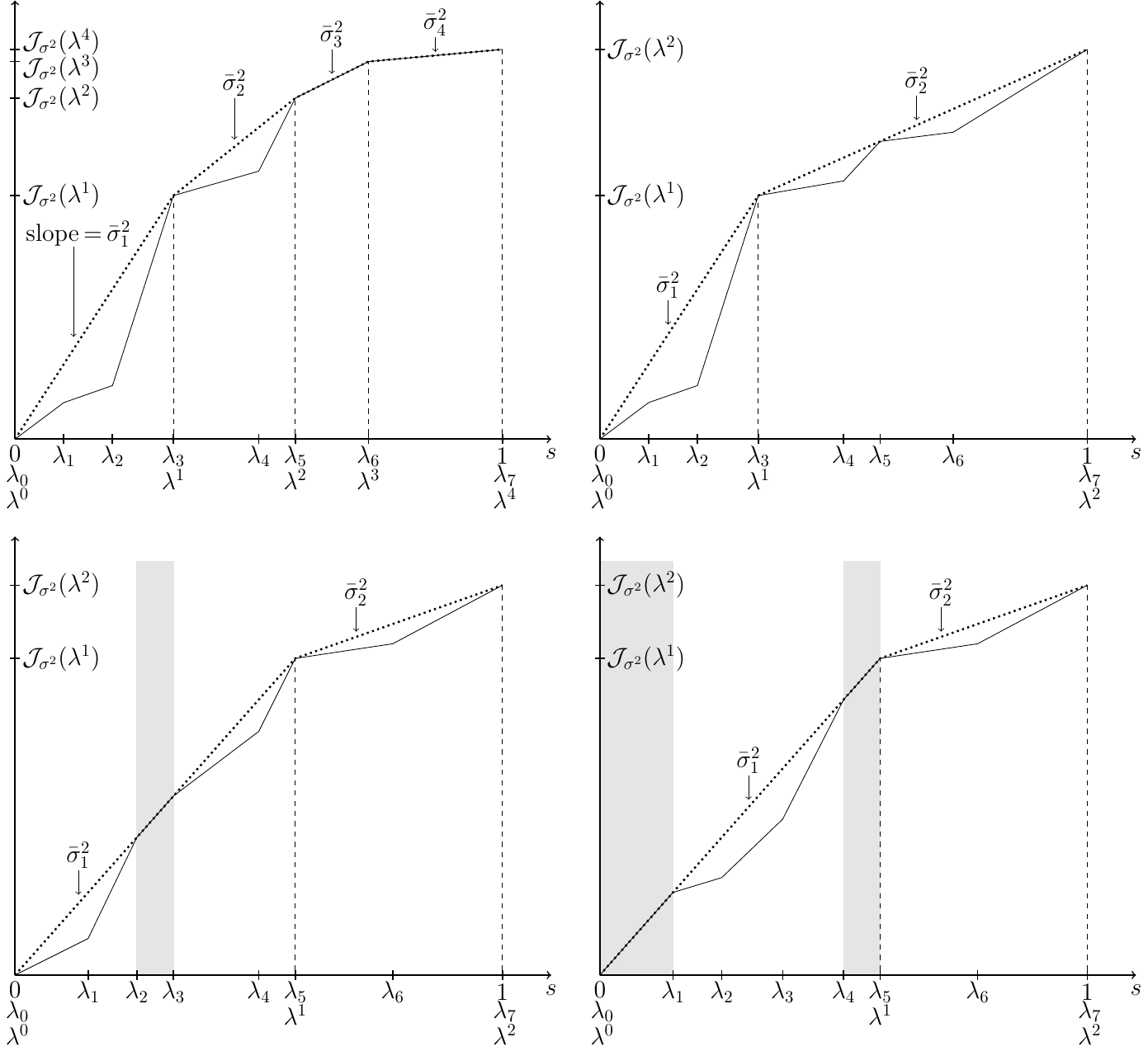}
        \captionsetup{width=0.85\textwidth}
        \caption{Examples of $\mathcal{J}_{\sigma^2}$ (closed lines) and $\hat{\mathcal{J}}_{\sigma^2}$ (dotted lines).}
        \label{fig:concave}
    \end{figure}

    It is easy to see that there exists a unique non-increasing left-continuous step function $s \mapsto \bar{\sigma}(s)$ such that
    \vspace{-1mm}
    \begin{equation*}
        \hat{\mathcal{J}}_{\sigma^2}(s) = \mathcal{J}_{\bar{\sigma}^2}(s) = \int_0^s \bar{\sigma}^2(r) dr \ \ \text{ for all $s\in (0,1]$.}\vspace{-1mm}
    \end{equation*}
    The scales in $[0,1]$ where $\bar{\sigma}$ jumps are denoted by
    \begin{equation}\label{eqn: lambda jumps}
        0 \circeq \lambda^0 < \lambda^1 < \ldots < \lambda^m \circeq 1,\vspace{-1mm}
    \end{equation}
    where $m \leq M$. As we will see in Theorem \ref{thm:IBRW.order2.restricted}, the \textit{effective scale parameters} $\lambda^j$ and the \textit{effective variance parameters} $\bar{\sigma}(\lambda^j)$ are the only parameters needed to fully determine the first and second order of the maximum for inhomogeneous branching random walks.

    To be consistent with previous notations, we set $\bar{\sigma}_j \circeq \bar{\sigma}(\lambda^j)$ and $t^j \circeq \lambda^j n$. We write $\nabla_{\hspace{-0.5mm}j}$ for the difference operator with respect to the index $j$. When the index variable is obvious, we omit the subscript. For example, $\nabla t^j = t^j - t^{j-1}$.

    To simplify the presentation of the proof of the main theorem, we impose a restriction on the variance parameters.
    \begin{restrict}\label{eq:IBRW.restriction}
        If $\mathcal{J}_{\sigma^2}$ and $\mathcal{J}_{\bar{\sigma}^2}$ coincide on a subinterval of $[\lambda^{j-1},\lambda^j]$ for some $j$, then they must coincide everywhere on $[\lambda^{j-1},\lambda^j]$.
    \end{restrict}

    \begin{remark}\label{rem:IBRW.restriction}
        Note that $\mathcal{J}_{\sigma^2}$ and $\mathcal{J}_{\bar{\sigma}^2}$ can still coincide at isolated points in $(\lambda^{j-1},\lambda^j)$ when they do not coincide everywhere in that interval.
        The union of all the scales $\lambda^j$ and all the isolated points where $\mathcal{J}_{\sigma^2}$ and $\mathcal{J}_{\bar{\sigma}^2}$ coincide form a subset of the scale parameters, say $\{\lambda_{i_d}\}_{0 \leq d \leq p}$, where $m \leq p \leq M$.
    \end{remark}
    \hspace{-5.5mm}
    For example, in Figure \ref{fig:concave}, the two models at the top satisfy Restriction \ref{eq:IBRW.restriction}, but the two models at the bottom do not. For the top models, the sets of scales described in Remark \ref{rem:IBRW.restriction} are respectively $\{\lambda_0,\lambda_3,\lambda_5,\lambda_6,\lambda_7\}$ and $\{\lambda_0,\lambda_3,\lambda_5,\lambda_7\}$.

    The main result of this paper is the derivation of the second order of the maximum (up to an $O_{\mathbb{P}}(1)$ error) for the $(\bb{\sigma},\bb{\lambda})$-BRW of Definition \ref{def:IBRW}, under Restriction \ref{eq:IBRW.restriction}. This was an open problem in \cite{MR2968674}.

    \begin{theorem}\label{thm:IBRW.order2.restricted}
        Let $\{S_v\}_{v\in \D_n}$ be as in Definition \ref{def:IBRW}, under Restriction \ref{eq:IBRW.restriction}.
        Let $g \circeq \sqrt{2 \log 2}$. For all $\varepsilon > 0$, there exists $K_{\varepsilon} > 0$ such that for all $n\in \N$,
        \begin{equation*}
            \mathbb{P}\Bigg(\Bigg| \max_{v\in \D_n} S_v - \sum_{j=1}^m \left[g \bar{\sigma}_j \nabla t^j - \frac{(1 + 2\cdot \delta_j) \bar{\sigma}_j}{2 g} \log(\nabla t^j)\right]\Bigg| \geq K_{\varepsilon}\Bigg) < \varepsilon,
        \end{equation*}
        where $\delta_j \circeq 1$ when $\mathcal{J}_{\sigma^2}$ and $\mathcal{J}_{\bar{\sigma}^2}$ coincide on $[\lambda^{j-1},\lambda^j]$, and $\delta_j \circeq 0$ otherwise.
    \end{theorem}

    This theorem was proved in \cite{MR2968674} for the case $M\hspace{-0.5mm}=\hspace{-0.5mm}2$ and $\lambda_1\hspace{-0.5mm}=\hspace{-0.5mm}1/2$. Note that Restriction \ref{eq:IBRW.restriction} is always satisfied when $M\hspace{-0.5mm}=\hspace{-0.5mm}2$.

    \subsection{Related works}\label{sec:related.works}

    The first order of the maximum (without restriction),
    \begin{equation*}
        \lim_{n\rightarrow \infty} \mathbb{P}\Bigg(\Big|\max_{v\in \D_n} S_v - \sum_{j=1}^m g \bar{\sigma}_j \nabla t^j\Big| > \varepsilon n\Bigg) = 0, \quad\forall \varepsilon > 0,
    \end{equation*}
    was proved in Section 2 of \cite{Ouimet2014master} for the $(\boldsymbol{\sigma},\boldsymbol{\lambda})$-BRW and in \cite{MR3541850} for the analogous model of scale-inhomogeneous Gaussian free field (GFF). The proofs rely on an analysis of so-called ``optimal paths'' showing where the maximal particle must be at all times with high probability. These paths were found by a first moment heuristic and the resolution of a related optimisation problem (using the Karush-Kuhn-Tucker theorem).

    The more involved question of finding the second order of the maximum was first solved by \cite{MR2968674} \hspace{0.3mm}for the case $M\hspace{-0.5mm}=\hspace{-0.5mm}2$ and $\lambda_1\hspace{-0.5mm}=\hspace{-0.5mm}1/2$, and later by \cite{MR3361256}, when the law of the increments changes a finite number of times but is not necessarily Gaussian. In his proof, Mallein develops a time-inhomogeneous version of the spinal decomposition for the BRW. The argument presented in this paper was first developed, without the knowledge of Mallein's results, in Section 2.4 of \cite{Ouimet2014master} and instead generalizes the approach of \cite{MR2968674}. The proof rely on the control of the increments of high points at every \textit{effective scale} $\lambda^j$.

    One shortfall of the spinal decomposition is that it completely relies on the presence of an \textit{exact} branching structure. Specifically, a crucial step in \cite{MR3361256} is the proof of a time-inhomogeneous version of the classical many-to-one lemma, which is a direct consequence of his comparison between the size-biased law of the BRW (the usual change of measure) and a certain projection of a law on the set of planar rooted marked trees with spine.

    In contrast, our method can be adapted to a number of cases where the branching structure is only \textit{approximate}.
    For instance, although no explicit proof is written down, it can be applied to prove the second order of the maximum for the scale-inhomogeneous GFF of \cite{MR3541850}.
    The model differs from the time-inhomogeneous BRW in two ways :
    \begin{enumerate}
        \item The branching structure is {\it approximate} in the sense that increments
            of the field that are below the {\it branching scale} are not perfectly correlated
            and they decorrelate smoothly near the branching scale.
        \item At a given scale, the covariance of the increments of the field decays near
            the boundary of the domain. In the context of BRWs, this means that at a given time,
            the law of each point process would depend on the position of the associated ancestors
            in the tree.
    \end{enumerate}
    The recent developments in the study of
    \begin{itemize}
        \item[$\bullet$] cover times (see e.g.\hspace{-0.3mm} \cite{MR3263552,arXiv:1709.08151,MR3129800,MR3602852,MR3126579,MR1998762,MR2123929,MR2206347,MR2946152,MR3178464,MR2912708,MR2921974});
        \item[$\bullet$] the extremes of the randomized Riemann zeta function on the critical line (see e.g.\hspace{-0.3mm} \cite{MR3619786,arXiv:1807.04860,arXiv:1706.08462,arXiv:1304.0677,doi:10.1214/18-ECP154,arXiv:1609.00027});
        \item[$\bullet$] the maxima of the Riemann zeta function on random intervals of the critical line (see e.g.\hspace{-0.3mm} \cite{arXiv:1612.08575,doi:10.1007/s00440-017-0812-y});
        \item[$\bullet$] the maxima of the characteristic polynomials of random unitary matrices (see e.g.\hspace{-0.3mm} \cite{MR3594368,arXiv:1607.00243,doi:10.1093/imrn/rnx033});
        \item[$\bullet$] etc.
    \end{itemize}
    show that approximate branching structures are present in a huge variety of models. Hence, the approach of this paper might become relevant in applications beyond the study of ``pure'' BRW.

    For other recent and relevant results on branching processes in time-inhomogeneous environments, the reader is referred to \cite{MR3164771,MR3351476,MR2070334,MR2070335,arXiv:1802.03034,MR2981635,MR3531703,MR3373310,arXiv:1507.08835,MR3731796}.

\section{Proof of the main result}

    \subsection{Preliminaries}

    For all $v\in \D_n$ and $k,l\in \{1,\ldots,M\}$, we can compute from Definition \ref{def:IBRW} :
    \vspace{-1.8mm}
    \begin{equation}\label{eq:IBRW.variance.increments}
        \var{}{S_v(t_l) - S_v(t_{k-1})} = \sum_{i=k}^l \sigma_i^2 \nabla t_i = \mathcal{J}_{\sigma^2}(\lambda_{k-1},\lambda_l)\, n.\vspace{-1.8mm}
    \end{equation}
    The variance of the increments in \eqref{eq:IBRW.variance.increments} will be used repeatedly during the proofs in conjunction with the following lemma.

    \begin{lemma}[Gaussian estimates, see e.g.\hspace{-0.3mm} \cite{MR2319516}]\label{lem:gaussian.estimates}
        Suppose that $Z\sim \mathcal{N}(0,\sigma^2)$ where $\sigma > 0$, then for all $z > 0$,
        \begin{equation*}
            \left(1 - \frac{\sigma^2}{z^2}\right) \frac{\sigma}{\sqrt{2\pi} z} e^{-\frac{z^2}{2 \sigma^2}} \leq \prob{}{Z \geq z} \leq \frac{\sigma}{\sqrt{2\pi} z} e^{-\frac{z^2}{2 \sigma^2}}.
        \end{equation*}
    \end{lemma}

    The particle achieving the maximum of the BRW at time $n$ act like a Brownian bridge around the maximum level on all the intervals $[t^{j-1},t^j]$ where $\mathcal{J}_{\sigma^2}(\cdot / n)$ and $\mathcal{J}_{\bar{\sigma}^2}(\cdot / n)$ coincide. The extra log terms in Theorem \ref{thm:IBRW.order2.restricted} (when $\delta_j = 1$) compensate for the ``cost'' of the Brownian bridge to stay below a certain logarithmic barrier.
    The sets $\mathcal{A}_l$ below identify the indices $j$ of these intervals up to scale $\lambda^l$. The sets $\mathcal{T}_l$ consist of the \textit{effective times} $t^j, ~1 \leq j \leq l$, and the integer times in $[t^{j-1},t^j], ~j\in \mathcal{A}_l$, where a Brownian bridge estimate will be needed. More precisely, for all $l\in \{1,\ldots,m\}$,
    \begin{align*}
        \mathcal{A}_l
        &\circeq \{j\in \{1,\ldots,l\} : \delta_j = 1\} \\
        &= \left\{j\in \{1,\ldots,l\} : \hspace{-2mm}
            \begin{array}{l}
                \left.\mathcal{J}_{\sigma^2}\right|_{[\lambda^{j-1},\lambda^j]} \equiv \left.\mathcal{J}_{\bar{\sigma}^2}\right|_{[\lambda^{j-1},\lambda^j]}
            \end{array}
            \hspace{-1.5mm}\right\}, \\
        \mathcal{T}_l
        &\circeq \{t^1,t^2,\ldots,t^l\} \cup \mathsmaller{\bigcup}_{j\in \mathcal{A}_l} \{t^{j-1},t^{j-1} + 1,\ldots,t^j\}.
    \end{align*}

    \hspace{-4.2mm}Let \hspace{-0.5mm}$\vartheta_k \hspace{-0.7mm}\in \hspace{-0.7mm}\{1,\ldots,m\}$ be the index such that $t^{\vartheta_k - 1} \hspace{-1.2mm}< \hspace{-0.3mm}k \hspace{-0.4mm}\leq \hspace{-0.4mm}t^{\vartheta_k}$\hspace{-0.5mm}. For all $k \hspace{-0.7mm}\in \hspace{-0.7mm}\{0,\ldots,n\}$, the \textbf{concave hull} of the \textit{optimal path for the maximum} is
    \vspace{-0.5mm}
    \begin{equation}\label{eq:IBRW.order2.optimal.path}
        M_n^{\star}(k) \circeq \sum_{j=1}^{\vartheta_k} \frac{(k \wedge t^j - t^{j-1})}{\nabla t^j} \left[g \bar{\sigma}_j \nabla t^j - \frac{(1 + 2\cdot \delta_j) \bar{\sigma}_j}{2 g}\log(\nabla t^j)\right]
    \end{equation}
    where $g \circeq \sqrt{2 \log 2}$, as in Theorem \ref{thm:IBRW.order2.restricted}.
    We refer the reader to \cite{Ouimet2014master} or \cite{MR3541850} for a first moment heuristic. Note that $M_n^{\star}$ and the optimal path coincide on $\mathcal{T}_m$, see Figure \ref{fig:frontier.order2} for an example of $M_n^{\star}$ under Restriction \ref{eq:IBRW.restriction}.
    
    \vspace{3mm}
    \begin{figure}[ht]
        \includegraphics[scale=0.63]{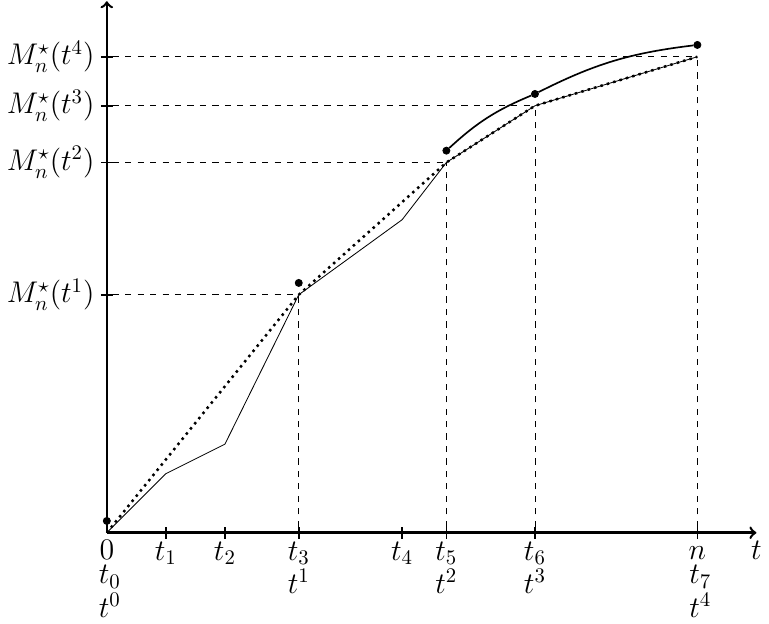}
        \captionsetup{width=0.85\textwidth}
        \caption{Example of the path $M_{n,x}^{\star}$ on the set $\mathcal{T}_m$ (in bold), the optimal path (thin line) and its concave hull $M_n^{\star}$ (dotted line).}
        \label{fig:frontier.order2}
    \end{figure}

    For all $k\in \mathcal{T}_m$, define the logarithmic barrier as
    \begin{equation}\label{eq:order2.barrier}
        b_n(k) \circeq \hspace{-0.3mm}
        \left\{\hspace{-1mm}
        \begin{array}{ll}
            \vspace{1mm}0 &\mbox{if } k\in \{t^0,t^1,\ldots,t^m\} \\
            \vspace{1mm}\frac{5}{2} \frac{\bar{\sigma}_{\vartheta_k}}{g} \log(k - t^{\vartheta_k - 1}) \hspace{-1mm}&\mbox{if } \vartheta_k\in \mathcal{A}_m, ~t^{\vartheta_k - 1} < k \leq \frac{t^{\vartheta_k - 1} + t^{\vartheta_k}}{2} \\
            \frac{5}{2} \frac{\bar{\sigma}_{\vartheta_k}}{g} \log(t^{\vartheta_k} - k) \hspace{-1mm}&\mbox{if } \vartheta_k\in \mathcal{A}_m, ~\frac{t^{\vartheta_k - 1} + t^{\vartheta_k}}{2} < k < t^{\vartheta_k}.
        \end{array}
        \right.
    \end{equation}
    For all $x > 0$, denote
    \begin{equation*}
        b_{n,x}(k) \circeq b_n(k) + x \qquad\text{and}\qquad
        M_{n,x}^{\star}(k) \circeq M_n^{\star}(k) + b_{n,x}(k).
    \end{equation*}
    Let us now define precisely what is meant by a Brownian bridge.
    \begin{definition}[Discrete Brownian bridge]\label{def:brownian.bridge}
        Let \hspace{0.3mm}$0 \leq \lambda < \lambda' \leq 1$ be such that $\lambda n,\lambda' n\in \N_0$ and $\sigma > 0$. The discrete $\sigma$-Brownian bridge on the interval $[\lambda n,\lambda' n]$ is a centered Gaussian vector $(B_k)_{k=\lambda n}^{\lambda' n}$ such that
        \begin{enumerate}[\ \ \ (a)]
            \item $B_{\lambda n} = B_{\lambda' n} = 0$,
            \item $\cov{B_k}{B_{k'}} = \frac{(k \wedge k' - \lambda n) (\lambda'n - k \vee k')}{(\lambda' - \lambda)n} \sigma^2, \ \ \, k,k' \hspace{-0.8mm}\in \hspace{-0.3mm}\{\lambda n,\lambda n + 1,\ldots,\lambda' n\}$. \vspace{3mm}
        \end{enumerate}
    \end{definition}

    Here are relevant examples of discrete Brownian bridges constructed from a discrete random walk.
    \begin{lemma}\label{lem:tech.lemma.2}
        Let $v\in \D_n$ and $j\in \mathcal{A}_m$. Then, the centered Gaussian vector
        \begin{equation}\label{eq:lem:tech.lemma.2.def.brownian.bridge}
            B_{v,i}^j \circeq S_v(i) - S_v(t^{j-1}) - \frac{i - t^{j-1}}{\nabla t^j} \nabla S_v(t^j), ~~~ t^{j-1} \leq i \leq t^j,
        \end{equation}
        is independent of $\{S_v(i')\}_{i'\not\in (t^{j-1},t^j)}$ \hspace{-0.3mm}and \hspace{-0.1mm}defines \hspace{-0.1mm}a \hspace{-0.1mm}discrete $\bar{\sigma}_j$-Brownian bridge under Definition \ref{def:brownian.bridge}.
        Similarly, when $l\in \mathcal{A}_m$ and $t^{l-1} \hspace{-0.5mm} < k \leq t^l$, the centered Gaussian vector
        \begin{equation}\label{eq:lem:tech.lemma.2.def.brownian.bridge.2}
            B_{v,i} \circeq S_v(i) - S_v(t^{l-1}) - \frac{i - t^{l-1}}{k - t^{l-1}} (S_v(k) - S_v(t^{l-1})), \ \ t^{l-1} \leq i \leq k,
        \end{equation}
        is independent of $\{S_v(i')\}_{i'\not\in (t^{l-1},k)}$ \hspace{-0.3mm}and defines a discrete $\bar{\sigma}_l$-Brownian bridge.
    \end{lemma}

    \begin{proof}
        We only prove \eqref{eq:lem:tech.lemma.2.def.brownian.bridge} since the proof of \eqref{eq:lem:tech.lemma.2.def.brownian.bridge.2} is totally analogous.
        Assume $j\in \mathcal{A}_m$, meaning that $\sigma(s) = \bar{\sigma}_j$ for all $s\in (\lambda^{j-1},\lambda^j]$.
        Then, for all $i'\in \{0,1,\ldots,t^{j-1}\} \cup \{t^j,t^j+1,\ldots,n\}$, $\text{Cov}\big(B_{v,i}^j,S_v(i')\big)$ is equal to
        \vspace{-1mm}
        \begin{align*}
            &\var{}{S_v(i \wedge i')} - \var{}{S_v(t^{j-1} \wedge i')} - \frac{i - t^{j-1}}{\nabla t^j} \nabla_{\hspace{-0.5mm}j} \hspace{0.3mm}\var{}{S_v(t^j \wedge i')} \notag \\
            &\stackrel{\phantom{\eqref{eq:IBRW.variance.increments}}}{=} \left\{\hspace{-1mm}
            \begin{array}{ll}
                \vspace{0.5mm}\var{}{S_v(i)} - \var{}{S_v(t^{j-1})} - \frac{i - t^{j-1}}{\nabla t^j} \nabla_{\hspace{-0.5mm}j} \hspace{0.3mm}\var{}{S_v(t^j)} &\mbox{if } t^j \leq i' \leq n \\
                0 - \frac{i - t^{j-1}}{\nabla t^j} 0 &\mbox{if } 0 \leq i' \leq t^{j-1}
            \end{array}\hspace{-1.7mm}
            \right\} \notag \\
            &\stackrel{\eqref{eq:IBRW.variance.increments}}{=} \left\{\hspace{-1mm}
            \begin{array}{ll}
                \vspace{0.5mm}\bar{\sigma}_j^2 (i - t^{j-1}) - \frac{i - t^{j-1}}{\nabla t^j} \bar{\sigma}_j^2 \nabla t^j &\mbox{if } t^j \leq i' \leq n \\
                0 &\mbox{if } 0 \leq i' \leq t^{j-1}
            \end{array}\hspace{-1.7mm}
            \right\} = 0.
        \end{align*}

        The first claim follows since $\{B_{v,i}^j\}_{i\in \{t^{j-1},\ldots,t^j\}}$ and $\{S_v(i')\}_{i'\not\in (t^{j-1},t^j)}$ form a Gaussian vector together.
        For the second claim, we need to verify $(a)$ and $(b)$ in Definition \ref{def:brownian.bridge} :

        \begin{itemize}
            \item[(a)] We obviously have $B_{v,t^{j-1}}^j = B_{v,t^j}^j = 0$ ;
            \item[(b)] For all $i,i'\in \{t^{j-1},t^{j-1}+1,\ldots,t^j\}$,
                \vspace{-3.5mm}
                \begin{align*}
                    \text{Cov}\big(B_{v,i}^j,B_{v,i'}^j\big)
                    &\stackrel{\phantom{\eqref{eq:IBRW.variance.increments}}}{=}\cov{S_v(i) - S_v(t^{j-1})}{S_v(i') - S_v(t^{j-1})} \\
                    &\hspace{8mm} - \frac{i - t^{j-1}}{\nabla t^j} \cov{\nabla S_v(t^j)}{S_v(i') - S_v(t^{j-1})} \\
                    &\hspace{14mm} - \frac{i' - t^{j-1}}{\nabla t^j} \cov{S_v(i) - S_v(t^{j-1})}{\nabla S_v(t^j)} \\
                    &\hspace{20mm} + \frac{(i - t^{j-1})(i' - t^{j-1})}{(\nabla t^j)^2} \var{}{\nabla S_v(t^j)} \\
                    &\stackrel{\eqref{eq:IBRW.variance.increments}}{=} (i \wedge i' - t^{j-1}) \bar{\sigma}_j^2 - 2\frac{(i - t^{j-1})(i' - t^{j-1})}{\nabla t^j} \bar{\sigma}_j^2 \\
                    &\hspace{8mm} + \frac{(i - t^{j-1})(i' - t^{j-1})}{(\nabla t^j)^2} \bar{\sigma}_j^2 \nabla t^j \\
                    &\stackrel{\phantom{\eqref{eq:IBRW.variance.increments}}}{=} \frac{(i \wedge i' - t^{j-1}) (t^j - i \vee i')}{\nabla t^j} \bar{\sigma}_j^2.
                \end{align*}\vspace{-7mm}
        \end{itemize}
        This ends the proof of the lemma.
    \end{proof}

    Finally, to estimate the probability that a discrete Brownian bridge stays below a logarithmic barrier such as the one defined in \eqref{eq:order2.barrier}, we adapt Proposition 1' of \cite{MR0494541}.

    \begin{lemma}[Discrete Brownian bridge estimates]\label{lem:bramson.pont.brownien}
        Let \hspace{0.3mm}$0 \hspace{-0.3mm}\leq \hspace{-0.3mm}\lambda \hspace{-0.3mm}< \hspace{-0.3mm}\lambda' \hspace{-0.3mm}\leq \hspace{-0.3mm}1$ be such that $\lambda n,\lambda' n\in \N_0$ and $\sigma > 0$. Let $(B_k)_{k=\lambda n}^{\lambda' n}$ be a discrete $\sigma$-Brownian bridge on the interval $[\lambda n,\lambda' n]$. For any constant $D = D(\lambda, \lambda',\sigma) > 0$ and the logarithmic barrier
        \begin{equation*}
            b(k) =
            \left\{\hspace{-1mm}
            \begin{array}{ll}
                0 &\mbox{if  }~ k\in \{\lambda n,\lambda' n\} \\
                D \log(k - \lambda n) &\mbox{if  }~ \lambda n < k \leq \frac{\lambda n + \lambda' n}{2} \\
                D \log(\lambda' n - k) &\mbox{if }~ \frac{\lambda n + \lambda' n}{2} < k < \lambda' n\, ,
            \end{array}
            \right.
        \end{equation*}
        there exists a constant $C = C(D,\sigma) \hspace{-0.3mm}> \hspace{-0.3mm}0$ such that for all $z \hspace{-0.3mm}> \hspace{-0.3mm}0$ and all $n \hspace{-0.3mm}\in \hspace{-0.3mm}\N$,
        \begin{equation*}
            \prob{}{B_k < b(k) + z, \ \lambda n \leq k \leq \lambda' n} \leq C \frac{(1+z)^2}{(\lambda' - \lambda) n}.
        \end{equation*}
    \end{lemma}

    In order to prove Lemma \ref{lem:bramson.pont.brownien}, we first need to prove that a random walk with Gaussian increments stays below the first part of the logarithmic barrier $b(\cdot) + z$ with probability $O(n^{-1/2})$. This is achieved through the following lemma, which is the analogue of Proposition 1 in \cite{MR0494541}.

    \begin{lemma}\label{lem:bramson.prop.1.analogue}
        Let $\sigma > 0$ and let $(S_k)_{k=0}^t$ be a discrete random walk with $\mathcal{N}(0,\sigma^2)$ increments and $S_0 \circeq 0$.
        For any constant $D = D(\lambda, \lambda',\sigma) > 0$ and the logarithmic barrier
        \begin{equation*}
            \widetilde{b}(k) =
            \left\{\hspace{-1mm}
            \begin{array}{ll}
                0 &\mbox{if  }~ k = 0 \\
                D \log k &\mbox{if  }~ 0 < k \leq t\, ,
            \end{array}
            \right.
        \end{equation*}
        there exists a constant $C = C(D,\sigma) \hspace{-0.3mm}> \hspace{-0.3mm}0$ such that for all $z \hspace{-0.3mm}> \hspace{-0.3mm}0$ and all $t \hspace{-0.3mm}\in \hspace{-0.3mm}\N$,
        \begin{equation*}
            \prob{}{S_k < \widetilde{b}(k) + z, \ 0 \leq k \leq t} \leq C \frac{(1+z)}{t^{1/2}}.
        \end{equation*}
    \end{lemma}

    \begin{remark}
        Throughout the proofs of this article, $c, C, \widetilde{C}, etc.,$ will denote positive constants whose value can change from line to line and can depend on the parameters $(\boldsymbol{\sigma},\boldsymbol{\lambda})$. For simplicity, equations are always implicitly stated to hold for $n$ large enough when needed.
    \end{remark}

    \begin{proof}
        Let $z > 0$ and $t\in \N$.
        When $t = 1$, the statement is trivially satisfied with $C \geq 1$.
        Therefore, assume $C \geq 1$ and $t \geq 2$ for the rest of the proof.
        Let $q_t = \lfloor D \log t\rfloor$ and for all $x > 0$, let $\tau_x \circeq \inf\{k\geq 1 : S_k \geq x\}$.
        Then,
        \begin{align}\label{eq:lem:bramson.prop.1.analogue.start}
            &\prob{}{S_k < \widetilde{b}(k) + z, \ 0 \leq k \leq t} \notag \\
            &~\leq \prob{}{\max_{0 \leq k \leq t} S_k < z} + \sum_{i=0}^{q_t} ~\prob{}{
                \begin{array}{l}\hspace{-1mm}
                    \lfloor e^{i / D} \rfloor \leq \tau_{z+i} \leq t ~~\text{and} \\
                    S_{\tau_{z+i}} \hspace{-0.5mm}< z+i+1 ~~\text{and} \\
                    \max_{\tau_{z+i} \leq k \leq t} (S_k - S_{\tau_{z+i}}) < 1
                \end{array}\hspace{-1.5mm}}.
        \end{align}
        We bound the first probability in \eqref{eq:lem:bramson.prop.1.analogue.start} using a standard gambler's ruin estimate.
        Indeed, from Theorem 5.1.7 in \cite{MR2677157}, there exists a constant $C' = C'(\sigma) > 0$ such that for all $z > 0$ and all $t\in \N$,
        \begin{equation}\label{eq:lem:bramson.prop.1.gambler.ruin}
            \prob{}{\max_{0 \leq k \leq t} S_k < z} \leq C' \frac{z + 1}{t^{1/2}}.
        \end{equation}
        We proceed to the individual summands in \eqref{eq:lem:bramson.prop.1.analogue.start}. The strong Markov property for the random walk implies
        \begin{align}\label{eq:lem:bramson.prop.1.analogue.middle.1}
            &\prob{}{
                \begin{array}{l}\hspace{-1mm}
                    \lfloor e^{i / D} \rfloor \leq \tau_{z+i} \leq t ~~\text{and} \\
                    S_{\tau_{z+i}} \hspace{-0.5mm}< z+i+1 ~~\text{and} \\
                    \max_{\tau_{z+i} \leq k \leq t} (S_k - S_{\tau_{z+i}}) < 1
                \end{array}\hspace{-1.5mm}} \notag \\
            &\quad= \hspace{-2mm}\sum_{j = \lfloor e^{i/D} \rfloor}^t \hspace{-2mm} \prob{}{\tau_{z+i} = j, ~S_{\tau_{z+i}} \hspace{-0.5mm}< z+i+1} \cdot \prob{}{\max_{0 \leq k \leq t-j} S_k < 1} \notag \\
            &\quad= \hspace{-2mm} \sum_{j = \lfloor e^{i/D}\rfloor \wedge (1+ \lfloor t/2 \rfloor)}^{\lfloor t/2 \rfloor} + \sum_{j = \lfloor e^{i/D}\rfloor \vee (1+ \lfloor t/2 \rfloor)}^t \hspace{-4mm}.
        \end{align}

        Now, for the first summation in \eqref{eq:lem:bramson.prop.1.analogue.middle.1}, we have
        \begin{align}\label{eq:lem:bramson.prop.1.analogue.middle.2}
            &\hspace{-4mm}\sum_{j = \lfloor e^{i/D}\rfloor \wedge (1+ \lfloor t/2 \rfloor)}^{\lfloor t/2 \rfloor} \hspace{-7mm} \prob{}{\tau_{z+i} = j, ~S_{\tau_{z+i}} \hspace{-0.5mm}< z+i+1} \cdot \prob{}{\max_{0 \leq k \leq t-j} S_k < 1} \notag \\
            &~\quad\quad\quad\leq \prob{}{\max_{0 \leq k \leq \lfloor e^{i/D} \rfloor} S_k < z + i + 1} \cdot \prob{}{\max_{0 \leq k \leq t - \lfloor t/2 \rfloor} S_k < 1} \notag \\
            &~\quad\quad\quad\leq C \frac{z+i+1}{t^{1/2}} e^{-i/(2D)}.
        \end{align}
        We applied the estimate \eqref{eq:lem:bramson.prop.1.gambler.ruin} to both terms on the second line and we used the fact that $(z + i + 2)/(z + i + 1) \leq 2$ for all $(z,i)\in (0,\infty) \times \N$ to obtain the last inequality.

        For the second summation in \eqref{eq:lem:bramson.prop.1.analogue.middle.1}, we can use an estimate closely related to the first hitting time distribution in the gambler's ruin problem. Indeed, from Lemma 3 in \cite{MR2599428}, there exists a constant $C'' = C''(\sigma) > 0$ such that for all $x > 0$ and all $j\in \N$,
        \begin{equation}\label{eq:lem:bramson.prop.1.hitting.time}
            \prob{}{\tau_x = j, ~S_{\tau_x} \hspace{-0.5mm}< x+1}
                \leq \prob{}{
                    \begin{array}{l}\hspace{-1mm}
                        S_j \in [x,x+1] ~\text{and} \\
                        \hspace{-1mm} S_j = \max_{0 \leq k \leq j} S_k
                    \end{array}\hspace{-1mm}} \leq C'' \frac{x + 1}{j^{3/2}}.
        \end{equation}
        Using successively \eqref{eq:lem:bramson.prop.1.hitting.time}, the gambler's ruin estimate \eqref{eq:lem:bramson.prop.1.gambler.ruin}, the change of variable $j' = t - j$ and the fact that $s \mapsto s^{-1/2}$ is decreasing, we have
        \begin{align}\label{eq:lem:bramson.prop.1.analogue.middle.4}
            &\sum_{j = \lfloor e^{i/D}\rfloor \vee (1+ \lfloor t/2 \rfloor)}^t \hspace{-7mm} \prob{}{\tau_{z+i} = j, ~S_{\tau_{z+i}} \hspace{-0.5mm}< z+i+1} \cdot \prob{}{\max_{0 \leq k \leq t-j} S_k < 1} \notag \\
            &~\quad\quad\leq \sum_{j = \lfloor e^{i/D}\rfloor \vee (1+ \lfloor t/2 \rfloor)}^t \hspace{-7mm} C'' \frac{z + i + 1}{j^{3/2}} \cdot \prob{}{\max_{0 \leq k \leq t-j} S_k < 1} \notag \\
            &~\quad\quad\leq 2^{3/2} C'' \frac{z + i + 1}{t^{3/2}} \cdot \left\{1 + \sum_{j' = 1}^{\lfloor t/2 \rfloor} C' \frac{2}{(j')^{1/2}}\right\} \notag \\
            &~\quad\quad\leq 2^{3/2} C'' \frac{z + i + 1}{t^{3/2}} \cdot \left\{4 (1 + C') \int_0^t \frac{1}{2 s^{1/2}} ds\right\} = C \frac{z + i + 1}{t}.
        \end{align}

        From \eqref{eq:lem:bramson.prop.1.analogue.middle.1}, \eqref{eq:lem:bramson.prop.1.analogue.middle.2} and \eqref{eq:lem:bramson.prop.1.analogue.middle.4}, we deduce
        \begin{equation}\label{eq:lem:bramson.prop.1.analogue.middle.5}
            \prob{}{
                \begin{array}{l}\hspace{-1mm}
                    \lfloor e^{i / D} \rfloor \leq \tau_{z+i} \leq t ~~\text{and} \\
                    S_{\tau_{z+i}} \hspace{-0.5mm}< z+i+1 ~~\text{and} \\
                    \max_{\tau_{z+i} \leq k \leq t} (S_k - S_{\tau_{z+i}}) < 1
                \end{array}\hspace{-1.5mm}} \leq C^{\star} \frac{z+i+1}{t^{1/2}} e^{-i/(2D)}
        \end{equation}
        for a certain constant $C^{\star}  = C^{\star}(\sigma) > 0$, since $t^{-1/2} \leq e^{-i/(2D)}$ for all $i \leq q_t$.

        Note that $(z + i + 1) \leq (z + 1) (i+1)$ for all $i\geq 0$. Therefore, by applying \eqref{eq:lem:bramson.prop.1.analogue.middle.5} and \eqref{eq:lem:bramson.prop.1.gambler.ruin} in \eqref{eq:lem:bramson.prop.1.analogue.start}, we get
        \begin{equation*}
            \prob{}{S_k < \widetilde{b}(k) + z, \ 0 \leq k \leq t} \leq C' \frac{z + 1}{t^{1/2}} + C^{\star} \frac{z+1}{t^{1/2}}\sum_{i=0}^{q_t} (i+1) e^{-i/(2D)}.
        \end{equation*}
        The conclusion holds since $\sum_{i=0}^{\infty} (i+1) e^{-i/(2D)} < \infty$.
    \end{proof}

    Now, the proof of Lemma \ref{lem:bramson.pont.brownien} is exactly the same (except in discrete time) as the proof of Proposition 1' in \cite{MR0494541} for the case $s_0 = t$. We give the details for completeness.

    \begin{proof}[Proof of Lemma \ref{lem:bramson.pont.brownien}]
        \hspace{-0.5mm}Without loss of generality, assume that $\lambda \hspace{-0.3mm}= \hspace{-0.3mm}0$, $\lambda' \hspace{-0.3mm}= \hspace{-0.3mm}1$ and $n/3 \hspace{-0.3mm}\in \hspace{-0.3mm}\N$.
        Let $(B_k)_{k=0}^n$ be a discrete $\sigma$-Brownian bridge and let $(S_k)_{k=0}^n$ be a discrete random walk with $\mathcal{N}(0,\sigma^2)$ increments and $S_0 \circeq 0$. Denote by $P_{b_1}$, $P_{b_2}$ and $P_{b_3}$, the sets of discrete paths in $\{0,1,\ldots,n\}$ lying below the barrier $b(\cdot) + z$ on the sets $\{0,\ldots,n/3\}$, $\{n/3,\ldots,2n/3\}$ and $\{2n/3,\ldots,n\}$ respectively.
        Using the Markov property of $B$ and $S$,
        \begin{align*}
            (\star)
            &\circeq \prob{}{B_k < b(k) + z, \ 0 \leq k \leq n} \\
            &= \int_{-\infty}^{b(n/3)+z} \hspace{-1mm}\int_{-\infty}^{b(2n/3)+z} \hspace{-4mm} \prob{}{B\in P_{b_1} | B_{n/3} = x_1} f_{B_{n/3}}(x_1) \\
            &\hspace{23mm} \times \prob{}{B\in P_{b_2} | B_{2n/3} = x_2, B_{n/3} = x_1} f_{B_{2n/3} | B_{n/3}}(x_2\hspace{0.2mm}|\hspace{0.2mm}x_1) \\
            &\hspace{23mm} \times \prob{}{B\in P_{b_3} | B_{2n/3} = x_2} dx_1 dx_2 \\
            &= \frac{1}{f_{S_n}(0)} \left\{\hspace{-1mm}
                \begin{array}{l}\hspace{-1mm}
                    \vspace{1mm} \int_{-\infty}^{b(n/3)+z} \int_{-\infty}^{b(2n/3)+z} \prob{}{S\in P_{b_1} | S_{n/3} = x_1} f_{S_{n/3}}(x_1) \\
                    \vspace{1mm} \times \prob{}{S\in P_{b_2} | S_{2n/3} = x_2, S_{n/3} = x_1} f_{S_{2n/3}|S_{n/3}}(x_2\hspace{0.2mm}|\hspace{0.2mm}x_1) \\
                    \times \prob{}{S\in P_{b_3} | S_n = 0, S_{2n/3} = x_2} f_{S_n|S_{2n/3}}(0\hspace{0.2mm}|\hspace{0.2mm}x_2) ~dx_1 dx_2
                \end{array}\hspace{-1mm}
                \hspace{-1mm}\right\}.
        \end{align*}
        But $\prob{}{S\in P_{b_2} | S_{2n/3} = x_2, S_{n/3} = x_1} \leq 1$, $f_{S_{2n/3}|S_{n/3}}(x_2\hspace{0.2mm}|\hspace{0.2mm}x_1) \leq f_{S_{n/3}}(0)$ and $f_{S_{n/3}}(0) / f_{S_n}(0) = \sqrt{3}$, so
        \begin{align*}
            (\star)
            &\leq \sqrt{3} \left\{\hspace{-1mm}
                \begin{array}{l}\hspace{-1mm}
                    \vspace{1mm} \int_{-\infty}^{b(n/3)+z} \prob{}{S\in P_{b_1} | S_{n/3} = x_1} f_{S_{n/3}}(x_1) d x_1 \\
                    \times \int_{-\infty}^{b(2n/3)+z} \prob{}{\hspace{-0.5mm}
                        S\in P_{b_3} \hspace{-0.5mm}\left|\hspace{-1.5mm}
                        \begin{array}{l}
                            S_n = 0, \\
                            S_{2n/3} = x_2
                        \end{array}
                        \right.\hspace{-2mm}}
                        f_{S_n|S_{2n/3}}(0\hspace{0.2mm}|\hspace{0.2mm}x_2) dx_2
                \end{array}\hspace{-1mm}
                \hspace{-1mm}\right\}.
        \end{align*}
        By the symmetry of $b(\cdot)$ around $n/2$, both integrals are exactly the same.
        Thus, the right-hand side is equal to
        \begin{equation*}
            \sqrt{3} \Big(\mathbb{P}\big(S_k < \widetilde{b}(k) + z, \ 0 \leq k \leq n/3\big)\Big)^2.
        \end{equation*}
        The conclusion follows directly from Lemma \ref{lem:bramson.prop.1.analogue}.
    \end{proof}

    \subsection{Why Restriction \ref{eq:IBRW.restriction} ?}\label{sec:discussion}

    Let $\pi_j\in \{0,1,\ldots,M\}$ denote the indice such that $\lambda_{\pi_j} = \lambda^j$.
    When the continuous and piecewise linear functions $\mathcal{J}_{\sigma^2}$ and $\mathcal{J}_{\bar{\sigma}^2}$ coincide on a subinterval of $[\lambda^{j-1},\lambda^j]$, they either coincide
    \begin{enumerate}
        \item everywhere on $[\lambda^{j-1},\lambda^j]$;
        \item everywhere on the left and right end, meaning on $[\lambda^{j-1},\lambda_{\pi_{j-1}+1}]$ and $[\lambda_{\pi_j-1},\lambda^j]$ respectively, but not somewhere in $(\lambda_{\pi_{j-1}+1},\lambda_{\pi_j-1})$;
        \item everywhere on the left end, but not on the right end;
        \item everywhere on the right end, but not on the left end.
    \end{enumerate}
    Imposing Restriction \ref{eq:IBRW.restriction} means that we only deal with the first case. The only reason we do this is to avoid overburdening the notation in the proof of Theorem \ref{thm:IBRW.order2.restricted} by dividing each interval $[t^{j-1},t^j]$, $j\in \mathcal{A}_m$, in three parts like we did in the proof of Lemma \ref{lem:bramson.pont.brownien}.

    From Lemma \ref{lem:bramson.prop.1.analogue}, the probability that the left (resp. right) end of a Brownian bridge stays below the left (resp. right) end of the logarithmic barrier $b(\cdot) + z$ is $O(n^{-1/2})$. The probability that the middle part of the Brownian bridge stays below the middle part of the logarithmic barrier is $O(1)$. Thus, it should now be obvious how to modify the statement of Theorem \ref{thm:IBRW.order2.restricted} when there is no restriction. Simply replace $2\cdot\delta_j$ by $\delta_j^{\text{left}} + \delta_j^{\text{right}}$, where
    \begin{align*}
        \delta_j^{\text{left}} &\circeq
        \left\{\hspace{-1mm}
        \begin{array}{ll}
            1 ~&\mbox{when } \mathcal{J}_{\sigma^2} ~\text{and } \mathcal{J}_{\bar{\sigma}^2} ~\text{coincide on } [\lambda^{j-1},\lambda_{\pi_{j-1}+1}] \\
            0 ~&\mbox{otherwise}
        \end{array}
        \right. \\
        \delta_j^{\text{right}} &\circeq
        \left\{\hspace{-1mm}
        \begin{array}{ll}
            1 ~&\mbox{when } \mathcal{J}_{\sigma^2} ~\text{and } \mathcal{J}_{\bar{\sigma}^2} ~\text{coincide on } [\lambda_{\pi_j - 1},\lambda^j] \\
            0 ~&\mbox{otherwise}.
        \end{array}
        \right.
    \end{align*}
    For confirmation, the reader is referred to Theorem 1.4 in \cite{MR3361256}, where a more general statement is given.

    \subsection{Second order of the maximum and tension}\label{sec:IBRW.order2}

    Theorem \ref{thm:IBRW.order2.restricted} is a direct consequence of Lemma \ref{lem:IBRW.order2.upper.bound.restricted}, which proves the exponential decay of the probability that the recentered maximum is above a certain level, and Lemma \ref{lem:IBRW.order2.lower.bound.restricted}, which shows the corresponding lower bound.

    \begin{lemma}[Upper bound]\label{lem:IBRW.order2.upper.bound.restricted}
        Let $\{S_v\}_{v\in \D_n}$ be the $(\bb{\sigma},\bb{\lambda})$-BRW at time $n$ of Definition \ref{def:IBRW}, \hspace{-0.25mm}under Restriction \ref{eq:IBRW.restriction}. \hspace{-0.25mm}Recall the definition of $M_n^{\star}$ from \eqref{eq:IBRW.order2.optimal.path}. There exists a constant $C = C(\bb{\sigma},\bb{\lambda}) > 0$ such that for all $x > 0$,
        \vspace{1.3mm}
        \begin{equation*}
            \prob{}{\max_{v\in \D_n} S_v \geq M_n^{\star}(n) + x} \hspace{-0.7mm}\leq  C \hspace{0.5mm}(1\hspace{-0.3mm} + \hspace{-0.3mm}x)^{2 \sum_{j=1}^m \delta_j} e^{-x \frac{g}{\bar{\sigma}_1}} \vspace{1.3mm}
        \end{equation*}
        for $n$ large enough, where $\delta_j \circeq \bb{1}_{\{j\in \mathcal{A}_m\}}$.
    \end{lemma}

    The proof of Lemma \ref{lem:IBRW.order2.upper.bound.restricted} is separated in two parts with a technical lemma in between them (Lemma \ref{lem:tech.lemma.1}).

    \begin{proof}[Proof of Lemma \ref{lem:IBRW.order2.upper.bound.restricted} (first part)]
        Define the set of particles that are above the path $M_{n,x}^{\star}$ at time $k$ :
        \begin{equation*}
            \mathcal{H}_{k,n,x} \circeq \{v\in \D_k : S_v(k) \geq M_{n,x}^{\star}(k)\}, \ \ k\in \mathcal{T}_m.
        \end{equation*}
        The idea of the proof is to split the probability that at least one particle at time $n$ exceeds $M_{n,x}^{\star}(n)$ by looking at the first time $k\in \mathcal{T}_m$ when the set $\mathcal{H}_{k,n,x}$ is non-empty. Using sub-additivity, we have the following upper bound on the probability of the lemma :
        \begin{align}\label{eq:lem:IBRW.order2.upper.bound.debut}
            \prob{}{|\mathcal{H}_{n,n,x}| \geq 1}
            &\leq \sum_{k\in \mathcal{T}_m} \prob{}{\hspace{-1mm}
                \begin{array}{l}
                    |\mathcal{H}_{k,n,x}| \geq 1 ~\text{and } |\mathcal{H}_{i,n,x}| = 0 \\
                    \forall i\in \mathcal{T}_m ~\text{such that } i<k
                \end{array}\hspace{-1mm}
                } \notag \\
            &\leq \sum_{k\in \mathcal{T}_m} 2^k \hspace{0.5mm}\max_{v\in \D_k} \hspace{0.5mm}\prob{}{\hspace{-1mm}
                \begin{array}{l}
                    S_v(k) \geq M_{n,x}^{\star}(k) \\
                    \text{and } S_v(i) < M_{n,x}^{\star}(i) \\
                    \forall i\in \mathcal{T}_m ~\text{such that } i<k
                \end{array}\hspace{-1mm}
                }.
        \end{align}
        {\fontdimen2\font=0.28em
        We only discuss the case $k > t^1$ from hereon. The case $k \leq t^1$ is easier (there is no conditioning in \eqref{eq:lem:IBRW.order2.upper.bound.integral}), \hspace{0.5mm}so we omit the details.
        \hspace{0.5mm}Fix \hspace{0.8mm}$l \hspace{-0.5mm}\in \hspace{-0.5mm}\{2,\ldots,m\}$ \hspace{0.5mm}and \hspace{0.7mm}$t^{l-1} \hspace{-0.8mm}< k \leq t^l$ \hspace{0.5mm}for the remaining of the proof.
        By conditioning on the event}
        \begin{equation*}
            E_v \circeq \{(S_v(t^1),\ldots,S_v(t^{l-1})) = (x_1,\ldots,x_{l-1}) \circeq \boldsymbol{x}\},
        \end{equation*}
        the probability in the maximum in \eqref{eq:lem:IBRW.order2.upper.bound.debut} is equal to
        \begin{equation}\label{eq:lem:IBRW.order2.upper.bound.integral}
            \int_{-\infty}^{M_{n,x}^{\star}(t^1)} \hspace{-4mm}\ldots \int_{-\infty}^{M_{n,x}^{\star}(t^{l-1})} \underbrace{\mathbb{P}\hspace{-0.5mm}\left(
                \left.\hspace{-1mm}
                \begin{array}{l}
                    S_v(k) \geq M_{n,x}^{\star}(k) \\
                    \text{and } S_v(i) < M_{n,x}^{\star}(i) \\
                    \forall i\in \mathcal{T}_m ~\text{such that } i < k
                \end{array}
                \hspace{-1mm}\right| E_v\hspace{-1mm}
                \right)}_{\circeq ~(\clubsuit)} \hspace{0.5mm}f_v(\boldsymbol{x}) ~d \boldsymbol{x},\vspace{-1mm}
        \end{equation}
        where $f_v$ is the density function of $(S_v(t^1),\ldots,S_v(t^{l-1}))$.

        Now, make the convenient change of variables
        \begin{equation*}
            Y_{v,j} \circeq \nabla S_v(t^j) - \nabla M_n^{\star}(t^j), \ \ \ j\in \{1,\ldots,l-1\}.
        \end{equation*}
        By the independence of the increments, the density of the vector $(S_v(t^j))_{j=1}^{l-1}$ is the product of the densities of the $Y_{v,j}$'s, namely
        \begin{equation*}
            f_v(\boldsymbol{x}) \circeq f_v(x_1,\ldots,x_{l-1}) = f_{Y_{v,1}}(y_1) \cdot \ldots \cdot f_{Y_{v,l-1}}(y_{l-1})\, .
        \end{equation*}
        Since $\var{}{Y_{v,j}} = \var{}{\nabla S_v(t^j)} = \bar{\sigma}_j^2 \nabla t^j$, we can bound each density :
        \begin{align*}
            f_{Y_{v,j}}(y_j)
            = \frac{e^{-\frac{\left(y_j + \nabla M_n^{\star}(t^j)\right)^2}{2 \bar{\sigma}_j^2 \nabla t^j}}}{\sqrt{2\pi} \sqrt{\bar{\sigma}_j^2 \nabla t^j}}
            &\leq C 2^{-\nabla t^j} \frac{e^{\frac{(1 + 2\cdot \delta_j)}{2} \log(\nabla t^j)}}{\sqrt{\nabla t^j}} e^{-y_j \frac{g}{\bar{\sigma}_j}} \notag \\
            &= C 2^{-\nabla t^j} (\nabla t^j)^{\delta_j} \hspace{0.5mm}e^{-y_j \frac{g}{\bar{\sigma}_j}}.
        \end{align*}
        We deduce that the integral in \eqref{eq:lem:IBRW.order2.upper.bound.integral} is smaller than
        \begin{equation}\label{eq:lem:IBRW.order2.upper.bound.integral.2}
            C 2^{-t^{l-1}} \int_{-\infty}^x \int_{-\infty}^{x - y_1} \hspace{-4mm}\ldots \int_{-\infty}^{x - \sum_{j=1}^{l-2} y_j} (\clubsuit) \cdot \prod_{j=1}^{l-1} (\nabla t^j)^{\delta_j} \hspace{0.5mm}e^{-y_j \frac{g}{\bar{\sigma}_j}} ~d \boldsymbol{y}.
        \end{equation}

        From Lemma \ref{lem:tech.lemma.2}, we know that for all $j\in \mathcal{A}_{l-1}$, the process
        \begin{equation}\label{eq:lem:IBRW.order2.upper.bound.brownian.bridge.1}
            B_{v,i}^j \circeq S_v(i) - S_v(t^{j-1}) - \frac{i - t^{j-1}}{\nabla t^j} \nabla S_v(t^j), \ \ t^{j-1} \leq i \leq t^j,
        \end{equation}
        is independent of $\{S_v(i')\}_{i'\not\in (t^{j-1},t^j)}$ and defines a discrete $\bar{\sigma}_j$-Brownian bridge. Similarly, when $l\in \mathcal{A}_m$, the process
        \begin{equation}\label{eq:lem:IBRW.order2.upper.bound.brownian.bridge.2}
            B_{v,i} \circeq S_v(i) - S_v(t^{l-1}) - \frac{i - t^{l-1}}{k - t^{l-1}} (S_v(k) - S_v(t^{l-1})), \ \ t^{l-1} \leq i \leq k,
        \end{equation}
        is independent of $\{S_v(i')\}_{i'\not\in (t^{l-1},k)}$ and defines a discrete $\bar{\sigma}_l$-Brownian bridge.

        Using the independence of $S_v(k) - S_v(t^{l-1})$ with respect to $(S_v(t^j))_{j=1}^{l-1}$ and the processes in \eqref{eq:lem:IBRW.order2.upper.bound.brownian.bridge.1} and \eqref{eq:lem:IBRW.order2.upper.bound.brownian.bridge.2}, we get
        \begin{align}\label{eq:lem:IBRW.order2.upper.bound.clubsuit}
            \quad(\clubsuit)
            &\leq \prob{}{S_v(k) - S_v(t^{l-1}) \geq M_{n,x}^{\star}(k) - x_{l-1}} \notag \\
            &\quad\times \hspace{-2mm}\prod_{j\in \mathcal{A}_{l-1}} \hspace{-2mm}\prob{}{\hspace{-1mm}
                \begin{array}{l}
                    B_{v,i}^j \hspace{-0.5mm}< \hspace{-0.5mm}M_{n,x}^{\star}(i) - x_{j-1} - \frac{i - t^{j-1}}{\nabla t^j} \nabla x_j \\
                    \text{for all } i ~\text{such that } \hspace{0.5mm}t^{j-1} \hspace{-0.5mm}< \hspace{-0.5mm}i \hspace{-0.5mm}< \hspace{-0.5mm}t^j
                \end{array}\hspace{-1mm}
                } \notag \\
            &\quad\times \prob{}{\hspace{-1mm}
                \begin{array}{l}
                    B_{v,i} \hspace{-0.5mm}< \hspace{-0.5mm}(M_{n,x}^{\star}(i) - x_{l-1}) - \frac{i - t^{l-1}}{k - t^{l-1}} (M_{n,x}^{\star}(k) - x_{l-1}) \\
                    \text{for all } i ~\text{such that } \hspace{0.5mm}t^{l-1} \hspace{-0.5mm}< \hspace{-0.5mm}i \hspace{-0.5mm}< \hspace{-0.5mm}k
                \end{array}\hspace{-1mm}
                }^{\hspace{-1mm}\boldsymbol{1}_{\{l\in \mathcal{A}_m\}}} \notag \\
            &\circeq (1) \times \hspace{-2mm}\prod_{j\in\mathcal{A}_{l-1}} \hspace{-2mm}(2)_j \times (3).
        \end{align}
        We bound $(1)$ using a Gaussian estimate, and $(2)_j$ and $(3)$ using the Brownian bridge estimates of Lemma \ref{lem:bramson.pont.brownien}.
        We pause the proof of Lemma \ref{lem:IBRW.order2.upper.bound.restricted} to state and prove these bounds in Lemma \ref{lem:tech.lemma.1}.
        \phantom\qedhere
    \end{proof}

    \begin{lemma}\label{lem:tech.lemma.1}
        Let $l\in \{2,\ldots,m\}$ and $t^{l-1} < k \leq t^l$.
        As in \eqref{eq:lem:IBRW.order2.upper.bound.integral}, we make the change of variables
        \begin{equation}\label{eq:change.of.variables}
            Y_{v,j} \circeq \nabla S_v(t^j) - \nabla M_n^{\star}(t^j), \ \ \ j\in \{1,\ldots,l-1\}.
        \end{equation}
        In \eqref{eq:lem:IBRW.order2.upper.bound.clubsuit}, there exist constants $C,D > 0$, only depending on $(\bb{\sigma},\bb{\lambda})$, such that for $n$ large enough,
        \begin{equation}\label{eq:bound.on.(1)}
            (1)\leq C 2^{-(k - t^{l-1})} h_l(k) (k - t^{l-1})^{\bb{1}_{\{l\in \mathcal{A}_m \, \text{and } (t^{l-1} + t^l)/2 < k \leq t^l\}}} e^{-\frac{x - \sum_{j=1}^{l-1} y_j}{g^{-1} \bar{\sigma}_l}}
        \end{equation}
        where
        \begin{equation*}
            h_l(k) \circeq \left\{\hspace{-1mm}
            \begin{array}{ll}
                (k - t^{l-1})^{-3/2} ~&\mbox{when } l\in \mathcal{A}_m ~\text{and } t^{l-1} < k \leq \frac{t^{l-1} + t^l}{2} \\
                (t^l - k)^{-5/2} ~&\mbox{when } l\in \mathcal{A}_m ~\text{and } \frac{t^{l-1} + t^l}{2} < k < t^l \\
                1 ~&\mbox{when } k = t^l,
            \end{array}\right.
        \end{equation*}
        and
        \begin{equation}\label{eq:bound.on.(2)j}
            (2)_j \leq C \frac{(1 + D + 2x - 2\sum_{j'=1}^{j-1} y_{j'} - y_j)^2}{\nabla t^j}, \ \ \ j\in \mathcal{A}_{l-1}\, ,
        \end{equation}
        and
        \begin{equation}\label{eq:bound.on.(3)}
            (3) \leq
            \left\{\hspace{-1.5mm}
            \begin{array}{ll}
                C \frac{(1 + D + 2x - 2\sum_{j'=1}^{l-2} y_{j'} - y_{l-1})^2}{k - t^{l-1}} &\mbox{if } l\in \mathcal{A}_m \hspace{0.5mm}\text{and } \frac{t^{l-1} + t^l}{2} < k \leq t^l \\
                1 &\mbox{otherwise}.
            \end{array}
            \right.
        \end{equation}
    \end{lemma}

    \begin{proof}[Proof \hspace{-0.3mm}of \hspace{-0.3mm}inequality \hspace{-0.3mm}\eqref{eq:bound.on.(1)}]
        \hspace{-1mm}
        Since $\var{}{S_v(k) - S_v(t^{l-1})} = (k - t^{l-1}) \bar{\sigma}_l^2$ when $k\in \mathcal{T}_m$, a Gaussian estimate yields
        \vspace{-1mm}
        \begin{align}\label{eq:lem:IBRW.order2.upper.bound.gaussian.estimate}
            (1)
            &\circeq \prob{}{S_v(k) - S_v(t^{l-1}) \geq M_{n,x}^{\star}(k) - x_{l-1}} \notag \\
            &\leq \frac{\sqrt{(k - t^{l-1}) \bar{\sigma}_l^2}}{\sqrt{2\pi} (M_{n,x}^{\star}(k) - x_{l-1})} e^{- \frac{(M_{n,x}^{\star}(k) - M_{n,x}^{\star}(t^{l-1}) + M_{n,x}^{\star}(t^{l-1}) - x_{l-1})^2}{2(k - t^{l-1}) \bar{\sigma}_l^2}}.
        \end{align}
        Use successively $x_{l-1} \leq M_{n,x}^{\star}(t^{l-1})$ from \eqref{eq:lem:IBRW.order2.upper.bound.integral}, the definition of $M_n^{\star}$ in \eqref{eq:IBRW.order2.optimal.path}, the fact that $b_{n,x}(k) \geq x$ and $x\mapsto (\log x)/x$ is decreasing for $x\geq e$, to show
        \begin{align}\label{eq:lem:IBRW.order2.upper.bound.gaussian.estimate.explain}
            &M_{n,x}^{\star}(k) - x_{l-1} \geq M_{n,x}^{\star}(k) - M_{n,x}^{\star}(t^{l-1}) \notag \\
            & \quad = g (k - t^{l-1}) \bar{\sigma}_l - \frac{(1 + 2 \cdot \delta_l) \bar{\sigma}_l}{2 g} \frac{(k - t^{l-1})}{\nabla t^l} \log(\nabla t^l) + b_{n,x}(k) - x \notag \\
            & \quad \geq g (k - t^{l-1}) \bar{\sigma}_l - \frac{(1 + 2 \cdot \delta_l) \bar{\sigma}_l}{2 g} \log(e \vee (k - t^{l-1})).
        \end{align}
        Plugging inequality \eqref{eq:lem:IBRW.order2.upper.bound.gaussian.estimate.explain} in \eqref{eq:lem:IBRW.order2.upper.bound.gaussian.estimate} and using the definition of $b_{n,x}$ from \eqref{eq:order2.barrier} and the fact that $M_{n,x}^{\star}(t^{l-1}) - x_{l-1} = x - \sum_{j=1}^{l-1} y_j$, we have
        \begin{align*}
            (1)
            &\leq C 2^{-(k - t^{l-1})} ~\frac{e^{\frac{(1 + 2 \cdot \delta_l)}{2}\log(e \vee (k - t^{l-1})) -\frac{b_{n,x}(k) - x}{g^{-1} \bar{\sigma}_l}}}{\sqrt{k - t^{l-1}}} ~e^{-\frac{M_{n,x}^{\star}(t^{l-1}) - x_{l-1}}{g^{-1} \bar{\sigma}_l}} \notag \\
            &\leq \widetilde{C} 2^{-(k - t^{l-1})} h_l(k) (k - t^{l-1})^{\bb{1}_{\{l\in \mathcal{A}_m ~\text{and } (t^{l-1} + t^l)/2 < k \leq t^l\}}} e^{-\frac{x - \sum_{j=1}^{l-1} y_j}{g^{-1} \bar{\sigma}_l}}
        \end{align*}
        where
        \begin{equation*}
            h_l(k) \circeq \left\{\hspace{-1mm}
            \begin{array}{ll}
                (k - t^{l-1})^{-3/2} ~&\mbox{when } l\in \mathcal{A}_m ~\text{and } t^{l-1} < k \leq \frac{t^{l-1} + t^l}{2} \\
                (t^l - k)^{-5/2} ~&\mbox{when } l\in \mathcal{A}_m ~\text{and } \frac{t^{l-1} + t^l}{2} < k < t^l \\
                1 ~&\mbox{when } k = t^l.
            \end{array}\right.
        \end{equation*}
        Note that the last inequality is an equality with $\widetilde{C} = C$ whenever $k - t^{l-1} \geq e$. When $k - t^{l-1}\in \{1,2\}$, taking $\widetilde{C} = e^{3/2} \cdot C$ is sufficient to ``absorb'' the terms that do not cancel out exactly.
    \end{proof}

    \begin{proof}[Proof of inequality \eqref{eq:bound.on.(2)j}]
        Let $j\in \mathcal{A}_{l-1}$ and define
        \begin{equation*}
            z_{i,j} \circeq M_{n,x}^{\star}(i) - x_{j-1} - \frac{i - t^{j-1}}{\nabla t^j} \nabla x_j, \ \ \ t^{j-1} < i < t^j.
        \end{equation*}
        We have
        \begin{equation*}
            \begin{aligned}
                z_{i,j}
                &= b_{n,x}(i) + M_n^{\star}(i) + \left\{\frac{i - t^{j-1}}{\nabla t^j} x_{j-1} + \frac{t^j - i}{\nabla t^j} x_j\right\} - x_{j-1} - x_j \\
                &= b_{n,x}(i) + \left[M_n^{\star}(i) - \frac{t^j - i}{\nabla t^j} M_n^{\star}(t^{j-1}) - \frac{i - t^{j-1}}{\nabla t^j} M_n^{\star}(t^j)\right] \\
                &\quad+ \left\{\frac{i - t^{j-1}}{\nabla t^j} (x_{j-1} - M_n^{\star}(t^{j-1})) + \frac{t^j - i}{\nabla t^j} (x_j - M_n^{\star}(t^j))\right\} \\
                &\quad- (x_{j-1} - M_n^{\star}(t^{j-1})) - (x_j - M_n^{\star}(t^j)).
            \end{aligned}
        \end{equation*}
        Now, bound the braces using $(x_{j-1} - M_n^{\star}(t^{j-1})) \vee (x_j - M_n^{\star}(t^j)) \leq x$ from the integration limits of $x_{j-1}$ and $x_j$ in \eqref{eq:lem:IBRW.order2.upper.bound.integral}. The quantity between the brackets is zero because $M_n^{\star}$ is affine on $[t^{j-1},t^j]$. Consequently,
        \vspace{1mm}
        \begin{align}\label{eq:lem:IBRW.order2.upper.bound.zij}
            z_{i,j}
            &\leq b_{n,x}(i) + x - (x_{j-1} - M_n^{\star}(t^{j-1})) - (x_j - M_n^{\star}(t^j)) \notag \\
            &\stackrel{\eqref{eq:change.of.variables}}{=} b_n(i) + 2x - \sum_{j'=1}^{j-1} y_{j'} - \sum_{j'=1}^j y_{j'}.
        \end{align}
        Since $(2)_j \circeq \mathbb{P}(B_{v,i}^j < z_{i,j}, \ t^{j-1} \hspace{-1mm}< \hspace{-0.5mm}i \hspace{-0.5mm}< \hspace{-0.5mm}t^j)$, where $B_v^j$ is a discrete $\bar{\sigma}_j$-Brownian bridge on $[t^{j-1},t^j]$, the conclusion follows from Lemma \ref{lem:bramson.pont.brownien} and \eqref{eq:lem:IBRW.order2.upper.bound.zij}.
    \end{proof}

    \begin{proof}[Proof of inequality \eqref{eq:bound.on.(3)}]
        Assume $l\in \mathcal{A}_m$ and $(t^{l-1}+t^l)/2 < k \leq t^l$. The other cases are trivial because $(3)$ is a probability.
        Now, define
        \begin{equation*}
            z_i \circeq (M_{n,x}^{\star}(i) - x_{l-1}) - \frac{i - t^{l-1}}{k - t^{l-1}} (M_{n,x}^{\star}(k) - x_{l-1}), \ \ \ t^{l-1} < i < k.
        \end{equation*}
        Similarly to the proof of \eqref{eq:bound.on.(2)j}, the path $M_n^{\star}$ is affine on $[t^{l-1},t^l] \supseteq [t^{l-1},k]$ and $x_{l-1} - M_n^{\star}(t^{l-1}) \leq x$ from the integration limits of $x_{l-1}$ in \eqref{eq:lem:IBRW.order2.upper.bound.integral}, so
        \begin{align}\label{eq:lem:IBRW.order2.upper.bound.zi}
            z_i
            &= b_{n,x}(i) - \frac{i - t^{l-1}}{k - t^{l-1}} b_{n,x}(k) - \frac{k - i}{k - t^{l-1}} (x_{l-1} - M_n^{\star}(t^{l-1})) \notag \\
            &\quad+ \left[M_n^{\star}(i) - \frac{k - i}{k - t^{l-1}} M_n^{\star}(t^{l-1}) - \frac{i - t^{l-1}}{k - t^{l-1}} M_n^{\star}(k)\right] \notag \\
            &= b_n(i) - \frac{i - t^{l-1}}{k - t^{l-1}} b_n(k) + \left(1 - \frac{i - t^{l-1}}{k - t^{l-1}}\right) x \notag \\
            &\quad+ \frac{i - t^{l-1}}{k - t^{l-1}} (x_{l-1} - M_n^{\star}(t^{l-1})) - (x_{l-1} - M_n^{\star}(t^{l-1})) \notag \\
            &\leq b_n(i) - \frac{i - t^{l-1}}{k - t^{l-1}} b_n(k) + x - \sum_{j'=1}^{l-1} y_{j'} .
        \end{align}
        In order to use Lemma \ref{lem:bramson.pont.brownien}, it remains to show that the first two terms in \eqref{eq:lem:IBRW.order2.upper.bound.zi} are bounded by an appropriate logarithmic barrier. Assume for now that $k \neq t^l$. There are three cases to consider.

        \noindent
        \vspace{1mm}

        \noindent
        \fbox{{\bf Case 1} : All $i$ such that $t^{l-1} < i \leq (t^{l-1} + k)/2 < (t^{l-1} + t^l)/2 < k < t^l$}\vspace{3mm}

        \noindent
        \vspace{2mm}
        Clearly,
        \vspace{-6mm}
        \begin{align}\label{eq:lem:IBRW.order2.upper.bound.pont.1}
            \hspace{10mm}b_n(i) - \frac{i - t^{l-1}}{k - t^{l-1}} b_n(k) \hspace{1mm}\leq\hspace{1mm} b_n(i) \stackrel{\eqref{eq:order2.barrier}}{=} \frac{5}{2} \frac{\bar{\sigma}_l}{g} \log(i - t^{l-1}).
        \end{align}

        \noindent
        \vspace{1mm}

        \noindent
        \fbox{{\bf Case 2} : All $i$ such that $t^{l-1} < (t^{l-1} + k)/2 < i \leq (t^{l-1} + t^l)/2 < k < t^l$}\vspace{3mm}

        \noindent
        Observe that $i - t^{l-1} \leq t^l - i$ and $t^l - k \leq k - t^{l-1}$ and $x\mapsto (\log x)/x$ is decreasing for $x \geq e$. Also, we have $(t^l - i) = (t^l - k) + (k - i) \leq 2 (t^l - k) (k - i)$ because $a + b \leq 2ab$ for $a,b \geq 1$. Using all this (in that order), we get
        \begin{align}\label{eq:lem:IBRW.order2.upper.bound.pont.2}
            b_n(i) - \frac{i - t^{l-1}}{k - t^{l-1}} b_n(k)
            &\stackrel{\eqref{eq:order2.barrier}}{=} \frac{5}{2} \frac{\bar{\sigma}_l}{g} \left\{\log\left(\frac{i - t^{l-1}}{t^l - k}\right) + \frac{k - i}{k - t^{l-1}} \log(t^l - k)\right\} \notag \\
            &\stackrel{\phantom{\eqref{eq:order2.barrier}}}{\leq} \frac{5}{2} \frac{\bar{\sigma}_l}{g} \left\{\log\left(\frac{t^l - i}{t^l - k}\right) + \log(e \vee (k - i))\right\} \notag \\
            &\stackrel{\phantom{\eqref{eq:order2.barrier}}}{\leq} \frac{5}{2} \frac{\bar{\sigma}_l}{g} \left\{\log2 + 2 \log(e \vee (k - i))\right\}.
        \end{align}

        \noindent
        \fbox{{\bf Case 3} : All $i$ such that $t^{l-1} < (t^{l-1} + t^l)/2 < i < k < t^l$}\vspace{3mm}

        \noindent
        By the same reasoning as in Case 2 (without $i - t^{l-1} \leq t^l - i$), we get
        \begin{align}\label{eq:lem:IBRW.order2.upper.bound.pont.3}
            b_n(i) - \frac{i - t^{l-1}}{k - t^{l-1}} b_n(k)
            &\stackrel{\eqref{eq:order2.barrier}}{=} \frac{5}{2} \frac{\bar{\sigma}_l}{g} \left\{\log\left(\frac{t^l - i}{t^l - k}\right) + \frac{k - i}{k - t^{l-1}} \log(t^l - k)\right\} \notag \\
            &\stackrel{\phantom{\eqref{eq:order2.barrier}}}{\leq} \frac{5}{2} \frac{\bar{\sigma}_l}{g} \left\{\log2 + 2 \log(e \vee (k - i))\right\}.
        \end{align}
        Finally, when $k = t^l$, the inequalities \eqref{eq:lem:IBRW.order2.upper.bound.pont.1}, \eqref{eq:lem:IBRW.order2.upper.bound.pont.2} and \eqref{eq:lem:IBRW.order2.upper.bound.pont.3} are trivial because $b_n(k) = 0$. Therefore, applying all three inequalities in \eqref{eq:lem:IBRW.order2.upper.bound.zi}, there exist appropriate constants $D, \widetilde{D} > 0$, depending only on $(\bb{\sigma},\bb{\lambda})$, for which
        \begin{align*}
            z_i
            &\leq
            \left\{\hspace{-2mm}
            \begin{array}{ll}
                \vspace{1mm}\widetilde{D} \log(i - t^{l-1}) + D + x - \sum_{j'=1}^{l-1} y_{j'} &\mbox{if } t^{l-1} < i \leq \frac{t^{l-1} + k}{2} \\
                \widetilde{D} \log(k - i) + D + x - \sum_{j'=1}^{l-1} y_{j'} &\mbox{if } \frac{t^{l-1} + k}{2} < i < k
            \end{array}
            \right. \notag \\
            &\leq
            \left\{\hspace{-2mm}
            \begin{array}{ll}
                \vspace{1mm}\widetilde{D} \log(i - t^{l-1}) + D + 2x - 2\sum_{j'=1}^{l-2} y_{j'} - y_{l-1} &\mbox{if } t^{l-1} < i \leq \frac{t^{l-1} + k}{2} \\
                \widetilde{D} \log(k - i) + D + 2x - 2\sum_{j'=1}^{l-2} y_{j'} - y_{l-1} &\mbox{if } \frac{t^{l-1} + k}{2} < i < k.
            \end{array}
            \right.
        \end{align*}
        We used $\sum_{j'=1}^{l-2} y_{j'} \leq x$ from the integration limits of $y_{l-2}$ in \eqref{eq:lem:IBRW.order2.upper.bound.integral.2} to get the last inequality.
        When $l\in \mathcal{A}_m$, recall that $(3) \circeq \mathbb{P}(B_{v,i} < z_i, \ t^{l-1} \hspace{-1mm}< \hspace{-0.5mm}i \hspace{-0.5mm}< \hspace{-0.5mm}k)$, where $B_v$ is a discrete $\bar{\sigma}_l$-Brownian bridge on $[t^{l-1},k]$. Applying Lemma \ref{lem:bramson.pont.brownien} yields the conclusion.
    \end{proof}

    \begin{proof}[Proof of Lemma \ref{lem:IBRW.order2.upper.bound.restricted} (last part)]
        By applying Lemma \ref{lem:tech.lemma.1} in \eqref{eq:lem:IBRW.order2.upper.bound.clubsuit}, the integral in \eqref{eq:lem:IBRW.order2.upper.bound.integral.2} is smaller than
        \begin{align}\label{eq:lem:IBRW.order2.upper.bound.integral.3}
            &C \hspace{0.5mm}2^{-k} \hspace{0.5mm}h_l(k) \hspace{0.5mm}e^{-x \frac{g}{\bar{\sigma}_l}} \hspace{-1mm}\int_{-\infty}^x \int_{-\infty}^{x - y_1} \hspace{-4mm}\ldots \int_{-\infty}^{x - \sum_{j=1}^{l-2} y_j} \hspace{-1mm}(1 + D + 2x - 2\sum_{j'=1}^{l-2} y_{j'} - y_{l-1})^{2 \cdot \delta_l} \notag \\
            & ~\times \left[\prod_{j\in \mathcal{A}_{l-1}} (1 + D + 2x - 2\sum_{j'=1}^{j-1} y_{j'} - y_j)^2\right] \cdot \prod_{j=1}^{l-1} e^{y_j \left[\frac{g}{\bar{\sigma}_l} - \frac{g}{\bar{\sigma}_j}\right]} ~d \boldsymbol{y}
        \end{align}
        for an appropriate constant $D = D(\bb{\sigma},\bb{\lambda}) > 0$. To obtain \eqref{eq:lem:IBRW.order2.upper.bound.integral.3}, the terms $(\nabla t^j)$ in \eqref{eq:lem:IBRW.order2.upper.bound.integral.2} canceled with the factors $1/(\nabla t^j)$ in \eqref{eq:bound.on.(2)j}, for all $j\in \mathcal{A}_{l-1}$. Similarly, the term $(k - t^{l-1})$ in \eqref{eq:bound.on.(1)} canceled with the factor $1/(k - t^{l-1})$ in \eqref{eq:bound.on.(3)}, when $l\in \mathcal{A}_m$ and $(t^{l-1} + t^l)/2 < k \leq t^l$.

        To bound the integral in \eqref{eq:lem:IBRW.order2.upper.bound.integral.3}, it is crucial to observe that the brackets in the exponentials are always strictly positive because $\bar{\sigma}_1 > \bar{\sigma}_2 > \ldots > \bar{\sigma}_m$ by definition. Denote these brackets by $\beta_{j,l}, ~ 1 \leq j \leq l-1$.
        We evaluate the integral iteratively. Note that $\sum_{j=1}^{l-2} y_j \leq x$ and $\sum_{j=1}^{l-3} y_j \leq x$
        from the integration limits of $y_{l-2}$ and $y_{l-3}$ in \eqref{eq:lem:IBRW.order2.upper.bound.integral.3}. By integrating by parts, it is easy to show that the first integral (from the interior) have the property
        \begin{align*}
            &\int_{-\infty}^{x - \sum_{j=1}^{l-2} y_j} (1 + D + 2x - 2\sum_{j'=1}^{l-2} y_{j'} - y_{l-1})^a ~e^{y_{l-1} \beta_{l-1,l}} ~d y_{l-1} \notag \\
            &\leq \frac{(a+1)!}{(1 \wedge \beta_{l-1,l})^{a+1}} ~(1 + D + 2x - 2\sum_{j'=1}^{l-3} y_{j'} - y_{l-2})^a ~e^{(x - \sum_{j=1}^{l-2} y_j) \beta_{l-1,l}}
        \end{align*}
        for any exponent $a\in \N_0$.
        Therefore, iterating this reasoning in \eqref{eq:lem:IBRW.order2.upper.bound.integral.3} gives
        \begin{align*}\eqref{eq:lem:IBRW.order2.upper.bound.integral.3}
            &\leq \widetilde{C} \hspace{0.5mm}2^{-k} \hspace{0.5mm}h_l(k) \hspace{0.5mm}e^{-x \frac{g}{\bar{\sigma}_l}}  \hspace{-0.5mm} \cdot \hspace{-0.5mm} (1 + D + x)^{2 \sum_{j=1}^l \delta_j} e^{x \sum_{j=1}^{l-1} \beta_{j,j+1}} \\
            &= \widetilde{C} \hspace{0.5mm}2^{-k} \hspace{0.5mm}h_l(k) \hspace{0.5mm}e^{-x \frac{g}{\bar{\sigma}_1}} \hspace{-0.5mm} \cdot \hspace{-0.5mm} (1 + D + x)^{2 \sum_{j=1}^l \delta_j}.
        \end{align*}
        Applying this bound in \eqref{eq:lem:IBRW.order2.upper.bound.debut} yields the conclusion since
        \begin{equation*}
            \sum_{l=1}^m \hspace{-2mm}\sum_{\hspace{1mm}\substack{k\in \mathcal{T}_m \\ \hspace{1mm}t^{l-1} < k \leq t^l}} \hspace{-3mm} h_l(k) < \infty.
        \end{equation*}
        This ends the proof of Lemma \ref{lem:IBRW.order2.upper.bound.restricted}.
    \end{proof}

    \begin{lemma}[Lower bound]\label{lem:IBRW.order2.lower.bound.restricted}
        Let $\{S_v\}_{v\in \D_n}$ be the $(\bb{\sigma},\bb{\lambda})$-BRW at time $n$ of Definition \ref{def:IBRW}, \hspace{-0.25mm}under Restriction \ref{eq:IBRW.restriction}. \hspace{-0.25mm}Recall the definition of $M_n^{\star}$ from \eqref{eq:IBRW.order2.optimal.path}. For all $\varepsilon > 0$, there exists $K_{\varepsilon} > 0$ such that for all $n\in \N$,
        \begin{equation*}
            \prob{}{\max_{v\in \D_n} S_v \leq M_n^{\star}(n) - K_{\varepsilon}} < \varepsilon.
        \end{equation*}
    \end{lemma}

    \begin{proof}
        Let $S_n^{\star} \circeq \max_{v\in \D_n} S_v$. From Theorem $1$ of \cite{MR3012090}, we know that the family $\{S_n^{\star} - \text{Med}(S_n^{\star})\}_{n\in \N}$ is tight, that is for all $\varepsilon > 0$, there exists $\widetilde{K}_{\varepsilon} > 0$ such that for all $n\in \N$,
        \begin{equation}\label{eq:lem:IBRW.order2.borne.inferieure.tension.mediane}
            \prob{}{|S_n^{\star} - \text{Med}(S_n^{\star})| \geq \widetilde{K}_{\varepsilon}} < \varepsilon.
        \end{equation}
        We claim that there exist $c,C > 0$ and $n_0,\widetilde{n}_0 \in \N$ such that
        \begin{equation}\label{eq:lem:IBRW.order2.borne.inferieure.probabilite}
            \hspace{-1.5mm}
            \left\{\hspace{-1.5mm}
                \begin{array}{l}
                    \prob{}{S_n^{\star} \geq M_n^{\star}(n) - C} \geq c \\
                    \text{for all } n\geq n_0
                \end{array}\hspace{-2mm}
            \right\}
            \Longrightarrow
            \left\{\hspace{-1.5mm}
                \begin{array}{l}
                    \text{Med}(S_n^{\star}) \geq M_n^{\star}(n) - C - \widetilde{K}_c \\
                    \text{for all } n \geq \widetilde{n}_0
                \end{array}\hspace{-2mm}
            \right\}.
        \end{equation}
        Otherwise, by \eqref{eq:lem:IBRW.order2.borne.inferieure.tension.mediane}, for each choice of $c,C > 0$, there would exist a subsequence $\{n_i\}_{i\in \N}$ such that
        \begin{equation*}
            c \leq \prob{}{S_{n_i}^{\star} \geq M_{n_i}^{\star}(n_i) - C} \leq \prob{}{S_{n_i}^{\star} \geq \text{Med}(S_{n_i}^{\star}) + \widetilde{K}_c} < c\hspace{0.3mm} ,
        \end{equation*}
        which is impossible. If the left side of \eqref{eq:lem:IBRW.order2.borne.inferieure.probabilite} was satisfied for some constants $c,C > 0$, we could define $K_{\varepsilon} \circeq \widetilde{K}_{\varepsilon} + C + \widetilde{K}_c$, and \eqref{eq:lem:IBRW.order2.borne.inferieure.tension.mediane} would give
        \begin{equation*}
            \prob{}{S_n^{\star} \leq M_n^{\star}(n) - K_{\varepsilon}} \leq \prob{}{S_n^{\star} \leq \text{Med}(S_n^{\star}) - \widetilde{K}_{\varepsilon}} < \varepsilon, \quad n \geq \widetilde{n}_0,
        \end{equation*}
        and the proof of the lemma would be over.

        To conclude, it remains to show the left side of \eqref{eq:lem:IBRW.order2.borne.inferieure.probabilite}. We now use Restriction \ref{eq:IBRW.restriction}. Recall from Remark \ref{rem:IBRW.restriction} that $\{\lambda_{i_d}\}_{0 \leq d \leq p}$ is the union of all the scales $\lambda^j$ and all the isolated points where $\mathcal{J}_{\sigma^2}$ and $\mathcal{J}_{\bar{\sigma}^2}$ coincide.
        By independence of the increments, the left side of \eqref{eq:lem:IBRW.order2.borne.inferieure.probabilite} is satisfied if there exist constants $c,C > 0$ such that
        \begin{equation}\label{eq:lem:IBRW.order2.borne.inferieure.condition}
            \prob{}{\hspace{-4mm}\max_{\hspace{4mm}v\in \D_{\nabla_{\hspace{-0.5mm}d} t_{i_d}}} \hspace{-3.5mm}S_v^{i_d} \geq \nabla_{\hspace{-0.5mm}d} \hspace{0.3mm}M_n^{\star}(t_{i_d}) - (C/p)\hspace{-0.5mm}} \geq c^{1/p}, \quad 1 \leq d \leq p,
        \end{equation}
        where each field $\{S_v^{i_d}\}_v$ consists of the end points of an inhomogeneous BRW on the time interval $[0,\nabla_{\hspace{-0.5mm}d} \hspace{0.5mm}t_{i_d}]$ with variance parameters given by the step function $s\mapsto \sigma(s)$ on $(\lambda_{i_{d-1}},\lambda_{i_d}]$.

        It suffices to show \eqref{eq:lem:IBRW.order2.borne.inferieure.condition} for the subinterval(s) $[t_{i_{d-1}},t_{i_d}] \subseteq [0,t^1]$ since we did not assume anything on the other intervals $[t^{j-1},t^j]$. When $1\in \mathcal{A}_m$, that is when there is only one variance parameter $\sigma_1 = \bar{\sigma}_1$ on $(0,\lambda^1]$, then \eqref{eq:lem:IBRW.order2.borne.inferieure.condition} follows from Theorem 3 of \cite{MR2537549} by choosing $C > 0$ large enough and $c > 0$ small enough. Since $M_n^{\star}(\cdot)$ is linear on $[0,t^1]$ and the argument presented below could be applied for each subinterval of the partition (independently of $d$), we can assume, without loss of generality, that $t_{i_1} = t^1$, namely that
        \begin{equation}\label{eq:hypothesis}
            \begin{array}{c}
                \mathcal{J}_{\sigma^2} ~\text{lies strictly below its concave hull } \mathcal{J}_{\bar{\sigma}^2} \\
                \text{everywhere on $(0,t^1)$.}
            \end{array}
        \end{equation}

        The usual trick to prove a lower bound in the BRW setting is the Paley-Zygmund inequality. If we naively try to apply the Paley-Zygmund inequality to the number of particles that stay above the optimal path, the method will not work because the correlations of the BRW inflate the second moment too much, see \eqref{eq:lem:IBRW.order2.borne.inferieure.ratio.esperances}. Instead, we need to add a barrier condition that eliminates the overly large number of particles that are too far off the optimal path during their lifetime. For simplicity, we omit the superscript $i_1$ for $S_v^{i_1}$ in the remaining of the proof. Define $S_v \circeq S_v(t^1)$ and let
        \begin{align*}
            &I_n \circeq [M_n^{\star}(t^1), M_n^{\star}(t^1) + 1], \\
            &I_{k,n}(x) \circeq [s_{k,n}(x) - f_{k,n},s_{k,n}(x) + f_{k,n}], \\
            &\mathcal{N}_n \circeq \# \{v\in \D_{t^1} : S_v\in I_n, S_v(k)\in I_{k,n}(S_v) \ \ \forall 0 < k < t^1\},
        \end{align*}
        where $s_{k,n}(x)$ is a path leading to $x\in \R$ and $f_{k,n}$ is a concave barrier. The definition we give to $s_{k,n}$ could seem strange at first, but is actually quite natural. It is argued in \cite{MR3541850} and proved in Appendix A of \cite{Ouimet2014master} that the log-number of particles that are above the path
        \begin{equation*}
            s_{k,n}(x) \circeq \frac{\mathcal{J}_{\sigma^2}(k/n)}{\mathcal{J}_{\sigma^2}(\lambda^1)} x, \quad 0 \leq k \leq t^1,
        \end{equation*}
        during their lifetime is asymptotically the same as the log-number of particles above $x$ at time $t^1$. In particular, for particles reaching $x = M_n^{\star}(t^1)$ at time $t^1$, this path is optimal (for the first order). The barrier is
        \begin{equation}\label{eq:lem:IBRW.order2.lower.bound.barrier}
            f_{k,n} \circeq
            \left\{\hspace{-1mm}
            \begin{array}{ll}
                C_{\hspace{-0.3mm}f} \hspace{0.5mm}(\mathcal{J}_{\sigma^2}(k/n) n)^{2/3} ~&\mbox{if } 0 \leq k \leq t_1 \\
                C_{\hspace{-0.3mm}f} \hspace{0.5mm}(\mathcal{J}_{\sigma^2}(k/n,\lambda^1)n)^{2/3} ~&\mbox{if } t_1 < k \leq t^1
            \end{array}
            \right.
        \end{equation}
        where the constant $C_{\hspace{-0.5mm}f} \hspace{-0.5mm} > \hspace{-0.5mm}0$ will be chosen large enough later in the proof.
        The exponent $2/3$ is not essential here (any exponent in $(1/2,1)$ works), but this definition is useful for the Gaussian estimates.

        Under assumption \eqref{eq:hypothesis}, the Paley-Zygmund inequality yields that the probability in \eqref{eq:lem:IBRW.order2.borne.inferieure.condition} (when $d = 1$) is bounded from below by
        \begin{equation}\label{eq:lem:IBRW.order2.borne.inferieure.ratio.esperances}
            \prob{}{\max_{v\in \D_{t^1}} S_v \geq M_n^{\star}(t^1)} \hspace{0.5mm}\geq~ \prob{}{\mathcal{N}_n \geq 1} \stackrel{\textbf{P-Z}}{\geq} \frac{\left(\esp{}{\mathcal{N}_n}\right)^2}{\esp{}{(\mathcal{N}_n)^2}}.
        \end{equation}
        To conclude, we show $\esp{}{\mathcal{N}_n} \geq c_{\star}$ and $\esp{}{(\mathcal{N}_n)^2} \leq (\esp{}{\mathcal{N}_n})^2 + (1 + C_{\star})\esp{}{\mathcal{N}_n}$ for some constants $c_{\star},C_{\star} > 0$.

        \subsection*{Lower bound on the first moment}

        By the linearity of expectation, we have the lower bound
        \begin{align}\label{eq:lem:IBRW.order2.borne.inferieure.borne.moment1}
            \esp{}{\mathcal{N}_n}
            &= 2^{t^1} \prob{}{S_v\in I_n, S_v(k)\in I_{k,n}(S_v) \ \ \forall 0 < k < t^1} \notag \\
            &= 2^{t^1} \prob{}{S_v\in I_n} \prob{}{S_v(k)\in I_{k,n}(S_v) \ \ \forall 0 < k < t^1} \notag \\
            &\geq c_{\star},
        \end{align}

        \newpage
        \hspace{-4.5mm}provided that there exist constants $c_1, c_2 > 0$ such that
        \begin{enumerate}[\ \ \ \ \ \ \ \ (1)]
            \item $S_v$ is independent of $\{S_v(k) - s_{k,n}(S_v)\}_{k=0}^{t^1}$,
            \item $2^{t^1} \prob{}{S_v\in I_n} \geq c_1$,
            \item $\prob{}{S_v(k)\in I_{k,n}(S_v) \ \ \forall 0 < k < t^1} \geq c_2$.
        \end{enumerate}

        To show $(1)$, observe that $\var{}{S_v(k)} = \mathcal{J}_{\sigma^2}(k/n) n$ and $\var{}{S_v} = \mathcal{J}_{\sigma^2}(\lambda^1) n$ from \eqref{eq:IBRW.variance.increments}, so the independence between $S_v(k)$ and $S_v - S_v(k)$ gives
        \begin{equation*}
            \cov{S_v}{S_v(k) - s_{k,n}(S_v)} = \var{}{S_v(k)} - \frac{\mathcal{J}_{\sigma^2}(k/n)}{\mathcal{J}_{\sigma^2}(\lambda^1)} \var{}{S_v} = 0.\vspace{-1mm}
        \end{equation*}
        To show $(2)$, note that $M_n^{\star}(t^1) = g \bar{\sigma}_1 t^1 - \frac{1}{2} \frac{\bar{\sigma}_1}{g} \log(t^1)$, under assumption \eqref{eq:hypothesis}, and $\var{}{S_v} = \bar{\sigma}_1^2 t^1$. Therefore,
        \begin{equation*}
            \prob{}{S_v\in I_n} = \int_{M_n^{\star}(t^1)}^{M_n^{\star}(t^1) + 1} \hspace{-1.5mm}\frac{e^{-\frac{z^2}{2 \bar{\sigma}_1^2 t^1}}}{\sqrt{2 \pi \bar{\sigma}_1^2 t^1}} \hspace{0.5mm}dz \geq 1 \cdot \frac{c}{\sqrt{t^1}} e^{-\frac{(M_n^{\star}(t^1) + 1)^2}{2 \bar{\sigma}_1^2 t^1}} \geq c_1 \hspace{0.5mm}2^{-t^1}.
        \end{equation*}
        To show $(3)$, note that $\cov{s_{k,n}(S_v)}{S_v(k) - s_{k,n}(S_v)} = 0$, by the independence in $(1)$, and thus
        \begin{align*}
            \var{}{S_v(k) - s_{k,n}(S_v)}
            &= \cov{S_v(k)}{S_v(k) - s_{k,n}(S_v)} \\
            &= \mathcal{J}_{\sigma^2}(k/n) n \left[1 - \frac{\mathcal{J}_{\sigma^2}(k/n)}{\mathcal{J}_{\sigma^2}(\lambda^1)} \right].
        \end{align*}
        Then, sub-additivity followed by Gaussian estimates yield
        \begin{align*}
            &\prob{}{\hspace{-1.2mm}
                \begin{array}{l}
                    S_v(k)\in I_{k,n}(S_v) \\
                    \forall 0 < k < t^1
                \end{array}
                \hspace{-1.5mm}}
            \geq 1 - 2 \sum_{k=1}^{t^1 - 1} \prob{}{S_v(k) - s_{k,n}(S_v) > f_{k,n}} \\
            &\quad\quad\quad\quad\quad\quad\quad\geq 1 - 2 \sum_{k=1}^{t^1 - 1} C \exp\left(-\frac{1}{2} \frac{(f_{k,n})^2}{\mathcal{J}_{\sigma^2}(k/n) n \left[1 - \frac{\mathcal{J}_{\sigma^2}(k/n)}{\mathcal{J}_{\sigma^2}(\lambda^1)} \right]}\right).
        \end{align*}
        By considering the cases $0 < k \leq t_1$ and $t_1 < k < t^1$ separately, the last sum is bounded from above by
        \begin{equation*}
            \sum_{k=1}^{t_1} C e^{- \frac{1}{2} C_{\hspace{-0.5mm}f}^2 \hspace{0.5mm}\sigma_1^{2/3} k^{1/3}} + \sum_{k=t_1 + 1}^{t^1 - 1} C e^{-\frac{1}{2} C_{\hspace{-0.5mm}f}^2 \min_{i\in \{2,3,\ldots,\pi_1\hspace{-0.3mm}\}} \sigma_i^{2/3} (t^1 - k)^{1/3}}\vspace{-1mm}.
        \end{equation*}
        For $C_{\hspace{-0.5mm}f}$ large enough, this is strictly smaller than $1/2$, independently of $n$, which proves $(3)$.

        \subsection*{Upper bound on the second moment}

        To estimate the second moment, we split $\esp{}{(\mathcal{N}_n)^2}$ according to the branching time $\rho(u,v) \circeq \max\{r\in \{0,1,\ldots,t^1\} : u_r = v_r\}$ of each pair of particles :
        \vspace{-1mm}
        \begin{equation*}
            \esp{}{(\mathcal{N}_n)^2} = \sum_{r=0}^{t^1} \sum_{\substack{u,v\in \D_{t^1} \\ \rho(u,v) = r}} \hspace{-1mm}\prob{}{\hspace{-1mm}
                \begin{array}{l}
                    S_u,S_v\in I_n ~\text{and } S_u(k)\in I_{k,n}(S_u), \\
                    S_v(k)\in I_{k,n}(S_v) ~\text{for all } 0 < k < t^1
                \end{array}
                \hspace{-2mm}}.\vspace{-2mm}
        \end{equation*}
        When $\rho(u,v)=0$, the processes $\{S_u(k)\}_k$ and $\{S_v(k)\}_k$ are independent. Therefore, in the case $r=0$, the second sum above is bounded by $(\esp{}{\mathcal{N}_n})^2$ by adding the missing terms. In the case $r = t^1$, the second sum is equal to $\esp{}{\mathcal{N}_n}$ because $u$ and $v$ coincide. In the remaining cases $0 < r < t^1$, the increment $S_v - S_v(r)$ is independent of $\{S_u(k)\}_k$, and $S_u(k) = S_v(k)$ for all $k \leq r$. Therefore, $\esp{}{(\mathcal{N}_n)^2}$ is bounded from above by
        \vspace{-1mm}
        \begin{align}\label{eq:lem:IBRW.order2.borne.inferieure.moment2}
            &\hspace{-4mm}\left(\esp{}{\mathcal{N}_n}\right)^2 + \esp{}{\mathcal{N}_n} + \sum_{r=1}^{t^1 - 1} \hspace{-2.5mm}\sum_{\hspace{2mm}\substack{u,v\in \D_{t^1} \\ \rho(u,v) = r}}
                \hspace{-2.5mm}\prob{}{\hspace{-1mm}
                    \begin{array}{l}
                        S_u\in I_n ~\text{and } S_u(k)\in I_{k,n}(S_u) \\
                        \text{for all } 0 < k < t^1
                    \end{array}
                    \hspace{-1.5mm}}
        \end{align}
        \vspace{-8mm}
        \begin{equation*}
            \hspace{50mm}\cdot \max_{x\in I_n}
                \hspace{0.5mm}\prob{}{\hspace{-1mm}
                    \begin{array}{l}
                        S_v - S_v(r)\in x - I_{r,n}(x)
                    \end{array}
                    \hspace{-1mm}}. \notag
        \end{equation*}
        There are at most $2^{t^1} \hspace{-1.5mm}\cdot 2^{t^1 - r}$ \hspace{-0.5mm}pairs $(u,v) \hspace{-0.5mm}\in \hspace{-0.5mm}\D_{t^1}^2$ \hspace{-0.7mm}with branching time equal to $r$, so the double sum in \eqref{eq:lem:IBRW.order2.borne.inferieure.moment2} is bounded from above by
        \vspace{-1mm}
        \begin{equation}\label{eq:lem:IBRW.order2.borne.inferieure.moment2.suite}
            \esp{}{\mathcal{N}_n} \times \sum_{r=1}^{t^1 - 1} 2^{t^1 - r} \max_{\substack{x\in I_n \\ v\in \D_{t^1}}} \underbrace{\prob{}{\hspace{-1mm}
                    \begin{array}{l}
                        S_v - S_v(r)\in x - I_{r,n}(x)
                    \end{array}
                    \hspace{-1mm}}}_{(\spadesuit)_r}.\vspace{-2mm}
        \end{equation}
        It remains to estimate the probabilities $(\spadesuit)_r$ in \eqref{eq:lem:IBRW.order2.borne.inferieure.moment2.suite}. From \eqref{eq:IBRW.variance.increments}, we know that $\var{}{S_v - S_v(r)} = \mathcal{J}_{\sigma^2}(r/n,\lambda^1) n$ for all $v\in \D_{t^1}$.

        In the case $0 < r \leq t_1$, we have $f_{r,n} = C_{\hspace{-0.5mm}f} \hspace{0.5mm}(\sigma_1^2 r)^{2/3}$. Thus, for $x\in I_n$,
        \vspace{-0.5mm}
        \begin{align}
            (\spadesuit)_r
            &= \int_{x - I_{r,n}(x)} \hspace{-0.5mm}\frac{e^{-\frac{1}{2} \frac{z^2}{\mathcal{J}_{\sigma^2}(r/n,\lambda^1) n}}}{\sqrt{2\pi \mathcal{J}_{\sigma^2}(r/n,\lambda^1) n}} \hspace{0.5mm}dz
            \leq 2 f_{r,n} \frac{e^{-\frac{1}{2} \frac{(M_n^{\star}(t^1) - s_{r,n}(M_n^{\star}(t^1)) - f_{r,n})^2}{\mathcal{J}_{\sigma^2}(r/n,\lambda^1) n}}}{\sqrt{\mathcal{J}_{\sigma^2}(r/n,\lambda^1) n}} \notag \\
            &\leq C ~r^{2/3} ~2^{- \frac{\mathcal{J}_{\sigma^2}(r/n,\lambda^1) t^1}{\mathcal{J}_{\sigma^2}(\lambda^1)}} ~\frac{e^{\frac{1}{2} \frac{\mathcal{J}_{\sigma^2}(r/n,\lambda^1)}{\mathcal{J}_{\sigma^2}(\lambda^1)} \log(t^1)}}{\sqrt{\mathcal{J}_{\sigma^2}(r/n,\lambda^1) n}} ~e^{\frac{C_{\hspace{-0.5mm}f} \hspace{0.3mm}(\sigma_1^2 r)^{2/3}}{g^{-1} \bar{\sigma}_1}} \label{eq:in.align} \\
            &\leq C ~r^{2/3} ~2^{-(t^1 - \eta_1 r)} ~e^{\widetilde{C} \hspace{0.3mm}r^{2/3}}. \label{eq:lem:IBRW.order2.borne.inferieure.borne.triangle.1}
        \end{align}
        To obtain the last bound, we use two crucial observations.
        Since the function $x\mapsto (\log x)/x$ is decreasing for $x \geq e$, the ratio of the exponential over the square root in \eqref{eq:in.align} is bounded by a constant independent of $r$ and $n$. Also, under assumption \eqref{eq:hypothesis} and for $0 < r \leq t_1$,
        \begin{equation*}
            \frac{\mathcal{J}_{\sigma^2}(r/n,\lambda^1) t^1}{\mathcal{J}_{\sigma^2}(\lambda^1)} = t^1 - \frac{\frac{1}{r/n} \mathcal{J}_{\sigma^2}(r/n)}{\frac{1}{\lambda^1} \mathcal{J}_{\sigma^2}(\lambda^1)} r = t^1 - \frac{\frac{1}{\lambda_1} \mathcal{J}_{\sigma^2}(\lambda_1)}{\frac{1}{\lambda^1} \mathcal{J}_{\sigma^2}(\lambda^1)} r \circeq t^1 - \eta_1 r
        \end{equation*}
        where $\eta_1 < 1$ independently of $r$ and $n$. See Figure \ref{fig:eta} below for an example.

        \vspace{2mm}
        \setlength{\belowcaptionskip}{0pt}
        \begin{figure}[ht]
            \includegraphics[scale=0.7]{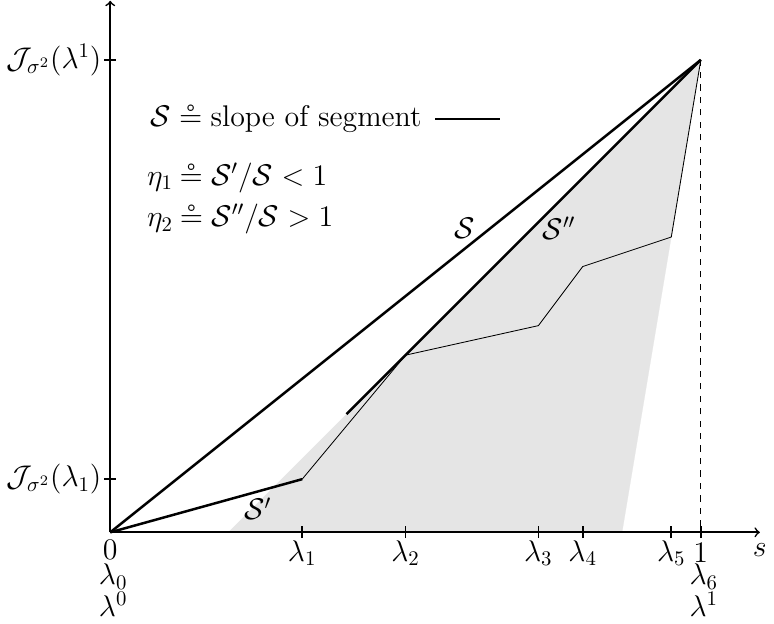}
            \captionsetup{width=0.8\textwidth}
            \caption{Example of $\eta_1$ and $\eta_2$ under assumption \eqref{eq:hypothesis}. The thin line represents $\mathcal{J}_{\sigma^2}$.}
            \label{fig:eta}
        \end{figure}

        Similarly, in the case $t_1 < r < t^1$, we have $f_{r,n} = C_{\hspace{-0.5mm}f} \hspace{0.5mm}(\mathcal{J}_{\sigma^2}(r/n,\lambda^1) n)^{2/3}$.
        Thus, for $x\in I_n$,
        \vspace{-3.5mm}
        \begin{align}
            (\spadesuit)_r
            &= \int_{x - I_{r,n}(x)} \hspace{-0.5mm}\frac{e^{-\frac{1}{2} \frac{z^2}{\mathcal{J}_{\sigma^2}(r/n,\lambda^1) n}}}{\sqrt{2\pi \mathcal{J}_{\sigma^2}(r/n,\lambda^1) n}} \hspace{0.5mm}dz
            \leq 2 f_{r,n} \frac{e^{-\frac{1}{2} \frac{(M_n^{\star}(t^1) - s_{r,n}(M_n^{\star}(t^1)) - f_{r,n})^2}{\mathcal{J}_{\sigma^2}(r/n,\lambda^1) n}}}{\sqrt{\mathcal{J}_{\sigma^2}(r/n,\lambda^1) n}} \notag \\
            &\leq C ~2^{- \frac{\mathcal{J}_{\sigma^2}(r/n,\lambda^1) t^1}{\mathcal{J}_{\sigma^2}(\lambda^1)}} ~\frac{e^{\frac{1}{2} \frac{\mathcal{J}_{\sigma^2}(r/n,\lambda^1)}{\mathcal{J}_{\sigma^2}(\lambda^1)} \log(t^1)}}{(\mathcal{J}_{\sigma^2}(r/n,\lambda^1) n)^{-1/6}} ~e^{\frac{C_{\hspace{-0.5mm} f} \hspace{0.3mm} (\mathcal{J}_{\sigma^2}(r/n,\lambda^1) n)^{2/3}}{g^{-1} \bar{\sigma}_1}} \label{eq:in.align.2} \\
            &\leq C ~2^{- \eta_2 (t^1 - r)} ~(\mathcal{J}_{\sigma^2}(r/n,\lambda^1) n)^{2/3} ~e^{\widetilde{C} \hspace{0.3mm}(t^1 - r)^{2/3}}. \label{eq:lem:IBRW.order2.borne.inferieure.borne.triangle.2}
        \end{align}
        Again, to obtain the last bound, we use two crucial observations. The first exponential in \eqref{eq:in.align.2} is bounded by $C (\mathcal{J}_{\sigma^2}(r/n,\lambda^1) n)^{1/2}$, where $C$ is independent of $r$ and $n$, using the fact that $x\mapsto (\log x)/x$ is decreasing for $x \geq e$. Also, under assumption \eqref{eq:hypothesis} and for $t_1 < r < t^1$,
        \begin{equation*}
            \frac{\mathcal{J}_{\sigma^2}(r/n,\lambda^1) t^1}{\mathcal{J}_{\sigma^2}(\lambda^1)} = \frac{\frac{1}{\lambda^1 - r/n} \mathcal{J}_{\sigma^2}(r/n,\lambda^1)}{\frac{1}{\lambda^1} \mathcal{J}_{\sigma^2}(\lambda^1)} (t^1 - r) \geq \eta_2 (t^1 - r)
        \end{equation*}
        where $\eta_2$ is the minimum of the last ratio with respect to $r\in \{t_1,\ldots,t^1-1\}$. Note that $\eta_2 > 1$ independently of $r$ and $n$, see Figure \ref{fig:eta} above.

        By combining the bounds on $(\spadesuit)_r$ in \eqref{eq:lem:IBRW.order2.borne.inferieure.borne.triangle.1} and \eqref{eq:lem:IBRW.order2.borne.inferieure.borne.triangle.2}, the sum in \eqref{eq:lem:IBRW.order2.borne.inferieure.moment2.suite} is bounded from above by
        \vspace{1mm}
        \begin{equation*}\label{eq:lem:IBRW.order2.borne.inferieure.borne.maltese}
            C \left[\sum_{r=1}^{t_1} 2^{-(1 - \eta_1) r + o(r)} + \sum_{r= t_1 + 1}^{t^1 - 1} 2^{(1 - \eta_2) (t^1 - r) + o(t^1 - r)}\right] \leq C_{\star}\vspace{2mm}
        \end{equation*}
        where $\eta_1 < 1$ and $\eta_2 > 1$ independently of $r$ and $n$. By applying this bound in \eqref{eq:lem:IBRW.order2.borne.inferieure.moment2.suite} and back in \eqref{eq:lem:IBRW.order2.borne.inferieure.moment2}, we have
        \vspace{1mm}
        \begin{equation}\label{eq:lem:IBRW.order2.borne.inferieure.borne.moment2}
            \frac{\esp{}{(\mathcal{N}_n)^2}}{\left(\esp{}{\mathcal{N}_n}\right)^2} \leq 1 + \frac{1 + C_{\star}}{\esp{}{\mathcal{N}_n}} \stackrel{\eqref{eq:lem:IBRW.order2.borne.inferieure.borne.moment1}}{\leq} 1 + \frac{1 + C_{\star}}{c_{\star}}.\vspace{1mm}
        \end{equation}
        Using \eqref{eq:lem:IBRW.order2.borne.inferieure.borne.moment2} in \eqref{eq:lem:IBRW.order2.borne.inferieure.ratio.esperances} yields \eqref{eq:lem:IBRW.order2.borne.inferieure.condition} when $d=1$, under assumption \eqref{eq:hypothesis}. This ends the proof of Lemma \ref{lem:IBRW.order2.lower.bound.restricted}.
    \end{proof}

\section*{Acknowledgements}

First, I would like to thank an anonymous referee for his valuable comments that led to improvements in the presentation of this paper.
I also gratefully acknowledge insightful discussions with Bastien Mallein and my advisor, Louis-Pierre Arguin.
This work is supported by a NSERC Doctoral Program Alexander Graham Bell scholarship (CGS D3).

%
%

\bibliographystyle{authordate1}
\bibliography{Ouimet_2018_BRW_bib}

\end{document}